\documentclass[10pt]{amsart}
\usepackage{amssymb}
\usepackage{amsmath,array,amsthm,mathtools,slashed,stmaryrd,tikz-cd, bbm, hyperref}
\tikzset{shorten <>/.style={shorten >=#1,shorten <=#1}}
\usepackage{tkz-euclide}
\usetikzlibrary{decorations.markings,arrows}
\usepackage[margin=1in]{geometry}

\RequirePackage{hyperref}
  \definecolor{link}{rgb}{0,0,0.5}
  \hypersetup{
     unicode,
     colorlinks=true,
     citecolor=link,
     filecolor=link,
     linkcolor=link,
     urlcolor=link
  }

\usepackage[capitalise]{cleveref}

\newtheorem{theorem}{Theorem}[section]
\newtheorem{lemma}[theorem]{Lemma}

\theoremstyle{definition}

\theoremstyle{remark}

\numberwithin{equation}{section}

\newcommand{\Bicat}{{\sf BiCat}}
\newcommand{\proj}{{\rm proj}}

\theoremstyle{plain}
\newtheorem{thm}{Theorem}[section]

\newtheorem{cor}[thm]{Corollary}
\newtheorem{lem}[thm]{Lemma}
\newtheorem{prop}[thm]{Proposition}
\newtheorem{defn}[thm]{Definition}
 
\theoremstyle{remark}
\newtheorem{rmk}[thm]{Remark}
\newtheorem{ex}[thm]{Example} 
\newtheorem{notation}[thm]{Notation}
\newtheorem{convention}[thm]{Convention}
 
\newtheorem*{thm*}{Theorem}

\newsavebox{\pullback}
\sbox\pullback{%
\begin{tikzpicture}%
\draw (0,0) -- (1ex,0ex);%
\draw (1ex,0ex) -- (1ex,1ex);%
\end{tikzpicture}}

\RequirePackage{mathrsfs}

\newcommand{\Cyl}{{\sf {Cyl}}}
\newcommand{\Fun}{{\sf {Fun}}}

\newcommand{\calg}{\mathcal{G}}

\newcommand{\calp}{\mathcal{P}}

\newcommand{\one}{\mathbbm{1}}

\newcommand{\cc}{{\sf c}}

\newcommand{\tri}{{\sf t}}

\newcommand{\beq}{\begin{eqnarray}}
\newcommand{\eeq}{\end{eqnarray}}
\newcommand{\bp}{\begin{proof}[Proof]}
\newcommand{\ep}{\end{proof}}

\newcommand{\rmH}{{\rm H}}

\newcommand{\Spin}{{\rm Spin}}
\newcommand{\String}{{\rm String}}
\newcommand{\U}{{\rm U}}

\newcommand{\SO}{{\rm SO}}

\newcommand{\Map}{{\sf Map}}

\newcommand{\cG}{\mathcal{G}}
\newcommand{\cZ}{\mathcal{Z}}
\newcommand{\cL}{\mathcal{L}}
\newcommand{\cP}{\mathcal{P}}
\newcommand{\cY}{\mathcal{Y}}

\newcommand{\sC}{\mathscr{C}}
\newcommand{\sL}{\mathscr{L}}
\newcommand{\sP}{\mathscr{P}}
\newcommand{\sPG}{\mathscr{P}_\cG}

\newcommand{\Gerbe}{{\sf Gerbe}}

\newcommand{\pt}{{*}}
\newcommand{\ev}{{\rm ev}}

\newcommand{\CP}{{\mathbb{CP}}}
\newcommand{\N}{{\mathbb{N}}}
\newcommand{\id}{{{\rm id}}}
\newcommand{\C}{{\mathbb{C}}}

\newcommand{\Z}{{\mathbb{Z}}}
\newcommand{\Aut}{{\sf Aut}}
\newcommand{\Bun}{{\sf Bun}}

\newcommand{\SL}{{\rm SL}}

\newcommand{\Hom}{{\sf Hom}}

\newcommand{\Grp}{{\sf Grp}}

\newcommand{\Diff}{{\sf Diff}}

\newcommand{\sq}{/\!\!/}

\newcommand\EquivTo{\xrightarrow{
   \,\smash{\raisebox{-0.65ex}{\ensuremath{\scriptstyle\sim}}}\,}}
\newcommand{\act}{\textsf{act}}
\newcommand{\gact}{\textrm{act}}

\newcommand{\FQ}{\mathcal{L}_{G}^\alpha}

\newcommand{\hol}{\text{hol}}
\newcommand{\eo}{\textsf{e}_1}
\newcommand{\et}{\textsf{e}_2}
\newcommand{\Stab}{\text{Stab}}
\newcommand{\Line}{\text{Line}}

\newcommand{\iin}{i_{\text{in}}}
\newcommand{\iout}{i_{\text{out}}}
\newcommand{\stabhom}[1]{{\sf R}_{#1}}

\newcommand{\BibunGX}{\Bun_\cG(X)}

\newcommand{\BibunGSig}{\Bun_\cG(\Sigma)}

\newcommand{\iso}{\varphi}
\newcommand{\Buntri}{\sf Bun_G(\Sigma)_\Delta}

\newcommand{\mf}\mathfrak
\newcommand{\mc}\mathcal
\newcommand{\mb}\mathbb

\newcommand{\rhohol}{{\tilde{\rho}}}
\newcommand{\gammahol}{\tilde{\gamma}}
\newcommand{\hhol}{\tilde{h}}
\newcommand{\etahol}{\tilde{\eta}}
\newcommand{\omegahol}{\tilde{\omega}}

\RequirePackage{todonotes}
\RequirePackage{cancel}
\usepackage[normalem]{ulem} 
\newcommand{\lout}{\bgroup\markoverwith
	{\textcolor{orange}{\rule[.6ex]{3pt}{0.6pt}}}\ULon}

\definecolor{OuthouseSurprise}{rgb}{0.498,0.6235,0.094}
\newcommand{\ECout}{\bgroup\markoverwith
	{\textcolor{OuthouseSurprise}{\rule[.6ex]{3pt}{0.6pt}}}\ULon}
\long\def\Etodo#1{{\color{OuthouseSurprise} {#1}}}

\title{The Freed--Quinn line bundle from higher geometry}
\author{Daniel Berwick-Evans, Emily Cliff, and Laura Murray}
\date{\today}

\begin{document}

\begin{abstract}
For a finite group $G$, and level $\alpha\in Z^3(BG;\U(1))$, Freed and Quinn construct a line bundle over the moduli space of $G$-bundles on surfaces. Global sections determine the values of Chern--Simons theory at level $\alpha$ on surfaces. In this paper, we provide an alternate construction using tools from higher geometry: the pair $(G,\alpha)$ determines a 2-group group $\cG$, and the Freed--Quinn line arises as a categorical truncation of the bicategory of $\cG$-bundles.  
\end{abstract}

\maketitle

\setcounter{tocdepth}{1}
\tableofcontents
\section{Introduction}

For an oriented surface $\Sigma$ and a finite group $G$, let $\Bun_G(\Sigma)$ denote the groupoid of principal $G$-bundles over~$\Sigma$. Given a 3-cocycle $\alpha\in Z^3(BG;\U(1))$ representing a class $[\alpha]\in \rmH^3(BG;\U(1))$, Freed and Quinn construct a line bundle $\FQ\to \Bun_G(\Sigma)$. The global sections of $\FQ$ define the value of Chern--Simons at level~$\alpha$ on the surface $\Sigma$. This paper provides a new construction of $\FQ$ using tools from higher geometry. 

Our construction begins with a categorical central extension~\cite{BL04,SP11}
\beq\label{Eq:catextension}
1\to \pt\sq \U(1)\to \cG \to G\to 1
\eeq
determined by the group $G$ and 3-cocycle $\alpha$, where $\cG$ is a \emph{2-group}, i.e., a monoidal groupoid where every object is (weakly) $\otimes$-invertible, see~\S\ref{sec:2grp}. Fixing a 2-group~\eqref{Eq:catextension} and oriented surface $\Sigma$, there is a bicategory $\Bun_\cG(\Sigma)$ of principal $\cG$-bundles on $\Sigma$ with a forgetful functor $\pi\colon \Bun_\cG(\Sigma)\to \Bun_G(\Sigma)$ to the usual groupoid of $G$-bundles on $\Sigma$. Taking fiberwise isomorphisms along $\pi$ we obtain 
\beq\label{eq:LGquotient}
\cP_\cG:=(\Bun_\cG(\Sigma)/{\rm fiberwise\ iso})\xrightarrow{[\pi]} \Bun_G(\Sigma),\qquad \cL_\cG:=\cP_\cG\times_{\U(1)}\C.
\eeq
Explicitly, the fibers of~\eqref{eq:LGquotient} are isomorphism classes of $\cG$-bundles with the same underlying $G$-bundle. We will show that $\cP_\cG$ is naturally a $\U(1)$-principal bundle, and $\cL_\cG$ is defined as the associated line bundle.

\begin{thm}\label{thm:main}
Fix a finite group $G$, degree 3 cocycle $\alpha\in Z^3(BG;\U(1))$, and oriented surface $\Sigma$. There is a canonical isomorphism of line bundles over~$\Bun_G(\Sigma)$ between the Freed--Quinn bundle and~\eqref{eq:LGquotient}, 
    \beq\label{eq:isowFQ}
\cL_{\cG}\simeq \FQ.
    \eeq
Furthermore, this isomorphism is equivariant for the action of the mapping class group of $\Sigma$ on $\Bun_G(\Sigma)$. 
\end{thm}

This result fits into a larger goal---initiated by Stolz and Teichner~\cite[\S5]{ST04}---to uncover interrelations between Chern--Simons theory, string structures, and equivariant elliptic cohomology. The finite group setting of this paper permits an entirely explicit investigation of structures that are expected to persist in the general case of a compact Lie group. We comment on this further in~\S\ref{sec:stringstructures} below.

\begin{rmk}\label{rmk:flatcG}
Throughout the paper, $\U(1)$ in~\eqref{Eq:catextension} is taken with the discrete topology, and hence the resulting $\cG$-bundles are always \emph{flat} $\cG$-bundles. 
\end{rmk}


As an application of Theorem~\ref{thm:main}, we connect the theory of 2-group bundles with the theory of Klein forms, a type of modular form with level structure~\cite[Chapter~XV]{Lang}. To set this up, consider a complex line bundle $L\to \Sigma$ of order $n$, i.e., $L^{\otimes n}\simeq\underline{\C}$ is trivialized. Hence, $L$ is classified by a map 
$$
\Sigma\xrightarrow{L} \pt\sq (\Z/n\Z)\hookrightarrow \pt\sq \U(1)
$$
where the second arrow is induced by including $\Z/n\Z\subset \U(1)$ as the $n$th roots of unity. Identifying $\U(1)\simeq \Spin(2)$, the string 2-group $\String(2)$ constructed by Schommer-Pries~\cite{SP11} determines a categorical central extension $\mathcal{Z}_n$ of $\Z/n\Z$ by pullback,
\beq\label{eq:finitestring}
\begin{tikzpicture}[baseline=(basepoint)];
\node (AA) at (-3,0) {$\pt\sq \U(1)$};
\node (CC) at (-3,-1) {$\pt\sq \U(1)$};
\node (A) at (0,0) {$\cZ_n$};
\node (B) at (3,0) {$\Z/n\Z$};
\node (C) at (0,-1) {$\String(2)$};
\node (D) at (3,-1) {$\U(1).$};
\draw[->] (A) to (B);
\draw[->] (B) to node [right] {$\mu_n$} (D);
\draw[->] (A) to (C);
\draw[->] (AA) to (A);
\draw[->] (CC) to (C);
\draw[->] (AA) to node [left] {$=$} (CC);
\draw[->] (C) to (D);
\path (0,-.75) coordinate (basepoint);
\end{tikzpicture}
\eeq
We refer to \cite[\S1]{BCMNP} for further discussion. Next we fix $\Sigma$ to be a genus~1 surface. There is a forgetful functor
\beq\label{eq:forgetcmplx}
\Bun_G(\mathcal{E})\to \Bun_G(\Sigma)
\eeq
from the moduli space of $G$-bundles over the universal elliptic curve $\mathcal{E}$ to the moduli space of $G$-bundle over a genus~1 surface. The functor~\eqref{eq:forgetcmplx} forgets from the complex analytic group structure on an elliptic curve to its underlying oriented manifold.

\begin{thm} \label{thm:Klein}
For $\cZ_n$ as in~\eqref{eq:finitestring}, the pullback of $\cL_{\cZ_n}$ along~\eqref{eq:forgetcmplx} is the line bundle whose sections are Klein forms. 
\end{thm}

Using analytic arguments, Freed identifies sections of Quillen's determinant line with Klein forms in \cite[Proposition~4.12]{Freed_Det}. Using Witten's description of Chern--Simons theory on a surface as sections of the determinant line, one can indirectly deduce Theorem~\ref{thm:Klein} from Theorem~\ref{thm:main}. However, we take a more concrete approach that computes the transformation properties of sections directly from the cocycle for the 2-group extension~\eqref{Eq:catextension}. These computational techniques can be applied to any finite 2-group, and hence this provides new calculation tools for line bundles over the moduli of elliptic curves constructed from 2-groups. Indeed, the line bundle in Theorem~\ref{thm:Klein}  is a specific case of a more general phenomenon, namely twists for equivariant elliptic cohomology from Chern--Simons theory~\cite{Grojnowski,Devoto,GanterHecke}. 

\begin{prop}\label{prop:Nora}
The pullback of $\cL_\cG$ along~\eqref{eq:forgetcmplx} determines the twisting for Ganter's $\alpha$-twisted $G$-equivariant elliptic cohomology~\cite[\S2]{GanterHecke}. 
\end{prop}

\bp
Ganter's twistings are defined in terms of the Freed--Quinn line bundle over the moduli stack of elliptic curves, so this follows immediately from Theorem~\ref{thm:main}. 
\ep

The construction~\eqref{eq:LGquotient} witnesses Ganter's twisting for equivariant elliptic cohomology as a categorical truncation of the moduli stack of 2-group bundles $\Bun_\cG(\mathcal{E})$ on elliptic curves, compare~\cite{Rezk}. This gives a potential inroad to a (higher) differential geometric counterpart to Lurie's 2-equivariant elliptic cohomology  ~\cite[\S5.5]{Lurie}. A better geometric understanding of 2-equivariance is a crucial step in one proposed approach to a geometric construction of elliptic cohomology and topological modular forms~\cite[\S1.3]{DBEsurvey}. Through the evident connections with Chern--Simons theory, Proposition~\ref{prop:Nora} also resonates with proposed generalization\Etodo{s} of equivariant elliptic cohomology twisted by a 3-dimensional topological field theory, e.g., see \cite{MTCEll,GKMP}.

\subsection{The key players}\label{sec:2gerbetriv}
For completeness, we begin with the classical definition of the moduli of $G$-bundles.

\begin{defn}\label{defn:BunG}
For a finite group $G$ and smooth manifold $X$, let $\Bun_G(X)$ denote the groupoid whose objects are principal $G$-bundles $P\to X$ and morphisms $G$-equivariant maps $\varphi\colon P\to P'$ covering the identity on~$X$. 
\end{defn}

Let $\Diff(X)$ denote the diffeomorphism group of $X$. Then $\Bun_G(X)$ has a  $\Diff(X)$-action via the pullback of $G$-bundles; see Definition~\ref{defn:weakaction} for group actions on categories. Our convention below is that for an oriented surface $\Sigma$, $\Diff(\Sigma)$ is the group of orientation-preserving diffeomorphisms.

A principal $G$-bundle $P\to X$ and 3-cocycle $\alpha\in Z^3(G;\U(1))$ determine maps $P$ and $\alpha$ 
\beq\label{eq:liftingcGbundle}
\begin{tikzpicture}[baseline=(basepoint)];
\node (B) at (0,0) {$X$};
\node (C) at (4,0) {$BG$};
\node (D) at (8,0) {$B^3\U(1)$};
\node (F) at (8,1) {$EB^2\U(1)$};
\node (E) at (4,1) {$B\cG$};
\node (AA) at (4.5,.7) {$\lrcorner$};
\draw[->] (B) to node [below] {$P$} (C);
\draw[->] (C) to node [below] {$\alpha$} (D);
\draw[->] (E) to (C);
\draw[->,dashed] (B) to (E);
\draw[->] (E) to (F);
\draw[->] (F) to (D);
\path (0,.5) coordinate (basepoint);
\end{tikzpicture}
\eeq
where we (somewhat abusively) define $B\mathcal{G}$ as the fibration over $BG$ classified by~$\alpha$, compare~\cite{BaezStevenson}. The diagram~\eqref{eq:liftingcGbundle} suggests that a principal $\cG$-bundle $\cP\to X$ is equivalent to the data of an ordinary $G$-bundle $P\to X$ and a trivialization of the 2-gerbe classified by the pullback of $[\alpha]\in \rmH^3(BG;\U(1))$ along the classifying map for $P$. There is a bit of work required to verify this expectation at the level of the bicategory $\Bun_\cG(X)$, but it is indeed the case~\cite[Theorem 1.2]{BCMNP}; see 
Remark~\ref{rmk:2gerbetriv2}. 

When $X=\Sigma$ is a surface, $\rmH^3(\Sigma;\U(1))\simeq \{1\}$ is the trivial group and hence any $G$-bundle $P\to \Sigma$ admits a lift to a $\cG$-bundle, as expected from~\eqref{eq:liftingcGbundle}. The moduli of such $\cG$-bundles is then the moduli of trivializations of the 2-gerbe determined by $P$ and $\alpha$, which in turn is a torsor over the symmetric monoidal bicategory of gerbes on $\Sigma$. 
In other words, the fibers of the forgetful functor $\pi$ carry a free and transitive action by the symmetric monoidal bicategory $\Gerbe_{\U(1)}(\Sigma)$ of 1-gerbes on $\Sigma$~\cite[Proposition 4.24]{BCMNP},
\beq\label{eq:highercataction}\label{Eq:forgetfulfunctor}
\pi\colon \Bun_{\cG}(\Sigma)\to \Bun_G(\Sigma), \qquad \Bun_{\cG}(\Sigma)\times \Gerbe_{\U(1)}(\Sigma)\to \Bun_{\cG}(\Sigma).
\eeq
Isomorphism classes of gerbes then act on a categorical truncation of $\Bun_\cG(\Sigma)$ via~\eqref{eq:highercataction}. As (flat) gerbes on an oriented surface $\Sigma$ are classified by $\rmH^2(\Sigma;\U(1))\simeq \U(1)$, such a categorical truncation gives a $\U(1)$-bundle on $\Bun_G(\Sigma)$ as a decategorification of~\eqref{eq:highercataction}. This completes the sketch of the construction of the $\U(1)$-bundle $\cP_\cG$ whose associated line bundle is $\cL_\cG$ in Theorem~\ref{thm:main}. The details are carried out in \S\ref{sec: construct P}.

\begin{rmk}
The situation~\eqref{Eq:forgetfulfunctor} is a higher categorical generalization of principal bundles whose structure groups are (ordinary) central extensions. The most common example is the groupoid of $\Spin^c$-structures for an $n$-dimensional bundle viewed as a principal bundle for the central extension,
$$
\U(1)\to \Spin^c(n)\to \SO(n).
$$
The groupoid of $\Spin^c$-structures is then a torsor over the symmetric monoidal category of complex lines: any pair of $\Spin^c$-structures differ by a hermitian line bundle. 
\end{rmk}

Next we review the Freed--Quinn line bundle. For a space $X$, consider the evaluation and projection maps
$$
BG \xleftarrow{\ev} X\times \Map(X,BG)\xrightarrow{\pi} \Map(X,BG).
$$
When $X$ is an oriented $n$-manifold, \emph{transgression} is the map in cohomology gotten from pulling back along evaluation and pushing forward along the projection 
\beq\label{eq:transgression}
\pi_! \circ \ev^* \colon \rmH^{n+k}(BG;\U(1))\to \rmH^k(\Map(X,BG);\U(1)).
\eeq
Freed and Quinn lift the cohomological map~\eqref{eq:transgression} to one at the level of geometric objects depending on~$n$ and~$k$. In particular, for an oriented surface $\Sigma$ and 3-cocycle $\alpha\in Z^3(G;\U(1))$, the pullback of $\alpha$ along evaluation determines a \emph{2-gerbe} which transgresses to a line bundle $\FQ$ on the groupoid $\Bun_G(\Sigma)$. This is compatible with classical transgression: the isomorphism class of $\FQ$ is the transgressed cohomology class
\beq\label{eq:FQclass}
[\FQ]=[\pi_!\circ \ev^*\alpha]\in \rmH^1(\Bun_G(\Sigma);\U(1))
\eeq
where we use that the groupoid $\Bun_G(\Sigma)$ provides one description of $\Map(\Sigma,BG)$. Global sections of $\FQ\to \Bun_G(\Sigma)$ are the value of Chern--Simons theory on the surface $\Sigma$ for the group $G$ and level $[\alpha]\in \rmH^3(BG;\U(1))$. 

\begin{rmk}
We make some technical remarks about how we compare the structures~\eqref{Eq:forgetfulfunctor} with Freed and Quinn's construction~\eqref{eq:transgression}. Freed and Quinn's construction of $\FQ$ relies on a specific presentation of the groupoid $\Bun_G(\Sigma)$ involving a triangulation of $\Sigma$ and a cell structure on~$BG$. We prove Theorem~\ref{thm:main} with a similarly concrete and combinatorial description of $\cG$-bundles. To this end, we express the higher geometric objects~\eqref{eq:highercataction} in terms of explicit \v{C}ech cocycles relative to an open cover of $\Sigma$. By completely general arguments, any groupoid presentation of $\Bun_G(\Sigma)$ leads to a cocycle description for a given line bundle over $\Bun_G(\Sigma)$. Computing the cocycle for $\cL_\cG$ in the presentation used by Freed and Quinn and seeing that it agrees with the cocycle for $\FQ$ gives a direct verification of Theorem~\ref{thm:main}. The key steps in this proof of Theorem~\ref{thm:main} use geometry:  Poincar\'e duality on the oriented surface $\Sigma$ and Stokes' Theorem for 3-manifolds with boundary, see~\S\ref{sec:proof of main thm}. 
One upshot of this approach is that other presentations of $\Bun_G(\Sigma)$ give different cocycle presentations of~$\cL_\cG$. Under~\eqref{eq:isowFQ}, this affords some new perspectives on the Freed--Quinn line $\FQ$, including a completely algebraic description in terms of categorical representation theory, see Theorem~\ref{thm1}. 
\end{rmk}

\subsection{String structures determine trivializations of Chern--Simons theory}\label{sec:stringstructures} 
Theorem~\ref{thm:main} provides a link between the definition of a generalized string structure as a higher principal bundle and the definition in terms of trivializations of Chern--Simons theory from~\cite[Definition 5.3.4]{ST04}, as we now explain. 

For a compact Lie group $G$, fix a smooth categorical central extension~\eqref{Eq:catextension} in the framework of~\cite{SP11}, and let $P\to M$ be a $G$-bundle over a manifold~$M$. This determines the maps $P$ and $\alpha$ in the diagram \beq\label{eq:CSstring}
\begin{tikzpicture}[baseline=(basepoint)];
\node (A) at (-.5,0) {$X$};
\node (B) at (2,0) {$M$};
\node (C) at (4,0) {$BG$};
\node (D) at (6,0) {$B^3\U(1),$};
\node (E) at (4,1) {$B\cG$};
\draw[->] (A) to node [below] {$\phi$} (B);
\draw[->] (B) to node [below] {$P$} (C);
\draw[->] (C) to node [below] {$\alpha$} (D);
\draw[->] (E) to (C);
\draw[->,dotted] (A) to (E);
\draw[->,dashed] (B) to (E);
\path (0,.5) coordinate (basepoint);
\end{tikzpicture}
\eeq
where $X$ is a $d$-manifold with $d\in \{0,1,2,3\}$. For $G=\Spin(n)$ and $\cG=\String(n)$, the dashed arrow $M\to B\cG$ in~\eqref{eq:CSstring} exists when $P\to M$ admits a string structure. It is also instructive to consider the more general situation as above for an arbitrary compact Lie group and categorical extension, following~\cite[\S5.4]{ST04}.

In the setting of~\eqref{eq:CSstring}, Stolz and Teichner define a (generalized) \emph{geometric string structure} as a compatible collection of dotted arrows determining trivializations of Chern--Simons theory on $\phi^*P\to X$ for each $\phi$ \cite[Definition 5.3.4]{ST04}. Specializing to $d=2$, this amounts to a trivialization of a line bundle over the mapping spaces $\Map(\Sigma,M)$ for $X=\Sigma$ a surface. This has been compared in \cite{Bunke} to Waldorf's definition of string structure \cite{Waldorfstring} as a trivialization of the 2-gerbe classified by~$P$ and~$\alpha$. However, so far these definitions have not been directly compared to Schommer--Pries's higher principal bundle definition of string structure~\cite{SP11} or to Stolz and Teichner's original definition as a trivialization of a 3-dimensional topological field theory. We note that the set of isomorphism classes is the same in all these examples, but showing that the various notions of string structure agree would require a lift of this bijection to an equivalence of bicategories.

Theorem~\ref{thm:main} allows us to make such comparisons in the special case that $X=\Sigma$ is a surface, $G$ is finite and the $\cG$-bundle is flat in the sense of Remark~\ref{rmk:flatcG}. In parallel to the discussion after~\eqref{eq:liftingcGbundle}, a lifting of the $G$-bundle $P\to M$ to a $\cG$-bundle is equivalent data to a trivialization of the 2-gerbe classifed by $P$ and $\alpha$. In this case, the line bundle $\cL_{\rm CS}$ that Chern--Simons theory determines over $\Map(\Sigma,M)$ is the pullback of the Freed--Quinn line,
\beq\label{eq:CSlineovermapping}
\begin{tikzpicture}[baseline=(basepoint)];
\node (A) at (0,0) {$\cL_{\rm CS}$};
\node (B) at (3,0) {$\cL_G^\alpha$};
\node (C) at (0,-1) {$\Map(\Sigma,M)$};
\node (D) at (3,-1) {$\Bun_G(\Sigma)$};
\draw[->] (A) to (B);
\draw[->] (B) to (D);
\draw[->] (A) to (C);
\draw[->] (C) to node [below] {$P\circ-$} (D);
\path (0,-.5) coordinate (basepoint);
\end{tikzpicture}
\eeq
where the lower horizontal arrow sends a map $\phi\colon \Sigma\to M$ to the principal bundle $\phi^*P\to \Sigma$.

\begin{cor}\label{cor:CS}
A trivialization of the 2-gerbe on $M$ determined by $P$ and $\alpha$ fixes a trivialization of the Freed--Quinn line over~$\Map(\Sigma,M)$. 
\end{cor}
\bp
A $\cG$-bundle lifting $P$ determines a compatible family of trivializations of the corresponding 2-gerbe. By Theorem~\ref{thm:main}, this in turn provides a section of the $\U(1)$-bundle whose associated line bundle is the Freed--Quinn line. Such a section is equivalent to a trivialization over $\Map(\Sigma,M)$. 
\ep

\begin{rmk}
The full comparison between flat string structures and Chern--Simons theory of a finite group requires an analysis of a fully-extended 3-dimensional field theory. The value of this theory on a 3-manifold is the Chern--Simons invariant (an element of $\U(1)$). The value of the theory on surfaces is the Freed--Quinn line. The value on 1-manifolds comes from a bundle of categories over the moduli of $G$ bundles on $S^1$, i.e., the adjoint quotient $G\sq G$. This is the twisted Drinfeld double of $G$ as described in \cite[Theorem 17]{willerton08}, and closely related to the twisted K-theory of $G\sq G$. The value on the point is some type of categorified twisted group ring \cite[\S4]{FHLT}. We note that in Stolz and Teichner's framework the value on the point is related to the hyperfinite $III_1$-factor, a type of von Neumann algebra. Many aspects of this framework remains mysterious, though recent progress provides a useful language in which to frame the problem of extending Chern--Simons theory down to points~\cite{Andrebicommutant}. 
\end{rmk}

\begin{rmk}
String structures viewed as trivializations of Chern--Simons theory serve as the basis for Stolz and Teichner's proposed geometric cocycle description of the Thom class in elliptic cohomology.
Hence, a fully-extended compact Lie generalization of Corollary~\ref{cor:CS} is a key step in the Stolz--Teichner program. 
\end{rmk}

\subsection{The geometry of 2-group principal bundles}
To explain some of the techiniques used in proving Theorem~\ref{thm:main}, we sketch a concrete description of the moduli of $\cG$-bundles over a surface~$\Sigma$ and the quotient~\eqref{eq:LGquotient}. For $\Sigma$ connected, we have the standard presentation of ordinary $G$-bundles
\beq\label{eq:holonomypresentation}
\Bun_G(\Sigma)\simeq \Hom(\pi_1\Sigma,G)\sq G\simeq \Fun(\pt\sq \pi_1\Sigma,\pt\sq G)
\eeq
 where in the middle description objects are homomorphisms $\rhohol\colon \pi_1\Sigma\to G$ corresponding to $G$-bundles of the specified holonomy with $G$ acting on such homomorphisms by conjugation. The right-most description in~\eqref{eq:holonomypresentation} considers the groupoid of functors and natural isomorphisms between the single-object groupoids $\pt\sq \pi_1\Sigma$ and $\pt\sq G$. The moduli of $\cG$-bundles enhances~\eqref{eq:holonomypresentation} as
 $$
 \Bun_\cG(\Sigma)\simeq \Hom(\pi_1\Sigma,\cG)\sq \cG \simeq \Fun(\pt\sq \pi_1\Sigma,\pt\sq \cG)
 $$
 where now $\Hom(\pi_1\Sigma,\cG)$ is the collection of monoidal functors, viewing $\pi_1\Sigma$ as a discrete groupoid with monoidal structure from group multiplication. Similarly, $\Fun(\pt\sq \pi_1\Sigma,\pt\sq \cG)$ is the bicategory whose objects are 2-functors
 $\pt\sq \pi_1\Sigma\to \pt\sq \cG$ between single-object bicategories that deloop the respective monoidal categories. Unpacking this, an object of $\Bun_\cG(\Sigma)$ is the data of a monoidal functor $\widehat{\rho}\colon \pi_1\Sigma\to \cG$, which we call a \emph{weak representation} of $\pi_1\Sigma$ valued in $\cG$. We describe $\widehat{\rho}$ in terms of more basic group theoretic data: on objects, $\widehat{\rho}$ determines an ordinary homomorphism $\rhohol$ as in~\eqref{eq:holonomypresentation}, and the data of $\widehat{\rho}$ as a monoidal functor provides a 2-cochain $\gammahol\colon G\times G\to \U(1)$ giving isomorphisms $\gamma(g,h)$ in $\cG$, 
\beq\label{eq:weakrep}
&&\rhohol\colon \pi_1\Sigma\to G={\rm Ob}(\cG),\qquad \gammahol(g,h)\colon \rhohol(gh)\xrightarrow{\sim} \rhohol(g)\rhohol(h), \quad g,h\in \pi_1\Sigma,
\eeq
where the above data satisfy the property $\rhohol^*\alpha=d\gammahol$ as 3-cocycles on $\pi_1\Sigma$. 

The standard presentation of the fundamental group of a genus ${\sf g}$ surface then allows one to express a weak representation $\widehat{\rho}\colon \pi_1\Sigma\to \cG$ in terms of $2{\sf g}$ elements of $G$ corresponding to the image of the generators of $\pi_1\Sigma$ in $\cG$, together with a morphism in $\cG$ that categorifies the relation inherited from $\pi_1\Sigma$, see Figure~\ref{fig:4gons}. For example, for $\Sigma=\mathbb{T}^2$ the torus, a homomorphism $\widehat{\rho}\colon \pi_1\mathbb{T}^2\to \cG$ the image of the generators of $\pi_1\mathbb{T}^2\simeq \Z^2$ are a pair of objects $g,h\in G={\rm Ob}(\cG)$, and $\widehat{\rho}$ further specifies an isomorphism 
\[
\sigma(g,h)\colon g\otimes h\xrightarrow{\sim} h\otimes g,
\]
witnessing the relation in $\pi_1\mathbb{T}^2$ that the generators commute. The datum $\sigma$ can be identified with an element of $\U(1)$, which turns out to be a complete isomorphism invariant of the weak representation $\widehat{\rho}$ lifting a fixed (ordinary) representation $\rho\colon \pi_1\Sigma\to G$. In other words, we obtain a characterization of isomorphism classes along the fibers of the forgetful functor~\eqref{Eq:forgetfulfunctor}.

\begin{thm}[Proposition~\ref{prop:weakrepn} and Remark~\ref{rmk:fundamental cycle}]\label{thm1}
Fixing a representation $\rhohol \colon \pi_1\Sigma\to G$, lifts to weak representations \begin{center}
    \begin{tikzcd}
        & \cG \arrow[d]\\
        \pi_1\Sigma \arrow[r] \arrow[ur, dashed]& G
    \end{tikzcd}
\end{center}
determine isomorphic $\cG$-bundles if and only if the categorified relations $\sigma$ corresponding to the weak representations as indicated in Figure~\ref{fig:4gons} are equal. 

\end{thm}

\begin{figure}\label{fig:visual of sigma}
    \centering
\begin{center}
\tikzset{
    arrowMe/.style={
        postaction=decorate,
        decoration={
            markings,
            mark=at position .55 with {\arrow[thick]{#1}}
        }
    }
}
\tikzset{
    arrowMeBig/.style={
        postaction=decorate,
        decoration={
            markings,
            mark=at position .75 with {\arrow[thick]{#1}}
        }
    }
}
\pgfarrowsdeclaredouble{<<s}{>>s}{stealth}{stealth}
\pgfarrowsdeclaretriple{<<<s}{>>>s}{stealth}{stealth}
\pgfarrowsdeclaredouble{<<<<s}{>>>>s}{<<}{>>}
\beq
\begin{tikzpicture}[baseline=(basepoint)];
   \coordinate (A) at (0,0);
   \coordinate (B) at (3,0);
   \coordinate (C) at (4,2);
   \coordinate (D) at (1,2);
   
    \tkzDrawSegments[arrowMe=stealth](B,A C,D);
    \tkzDrawSegments[>=stealth, arrowMe=>>](C,B D,A);
    
    \draw (A) -- node[below]{\tiny $g$} (B);
    \draw (D) -- node[above]{\tiny $g$} (C);
    \draw (B) -- node[right]{\tiny $h$} (C);
    \draw (A) -- node[left]{\tiny $h$} (D); 
    
    \draw[double,thick, -stealth,shorten >= 30pt, shorten <=30pt] (D) -- node[right=3pt]{\tiny $\sigma(g,h)$}(B);
    \path (0,1) coordinate (basepoint);
\end{tikzpicture}\qquad\qquad
\begin{tikzpicture}[baseline=(basepoint),scale=.6, transform shape]
   \coordinate (A) at (2.41, 1);
   \coordinate (B) at (1, 2.41);
   \coordinate (C) at (-1, 2.41);
   \coordinate (D) at (-2.41, 1);  
   \coordinate (E) at (-2.41, -1);
   \coordinate (F) at (-1, -2.41);
   \coordinate (G) at (1, -2.41);
   \coordinate (H) at (2.41, -1);
   
  \coordinate (M) at (1.705, -1.705);
  \coordinate (N) at (-1.705, 1.705);
   
  \tkzDrawSegments[arrowMe=stealth](A,B D,C);
  \draw (A) -- node[above right]{$g_1$} (B);
  \draw (D) -- node[above left]{$g_1$} (C);
  
  \tkzDrawSegments[>=stealth, arrowMe=>>](B,C E,D);
  \draw (B) -- node[above]{$h_1$} (C);
  \draw (E) -- node[left]{$h_1$} (D);
  
  \tkzDrawSegments[>=stealth, arrowMeBig=>>>](E,F H,G);
  \draw (E) -- node[below left]{$g_2$} (F);
  \draw (H) --node[below right]{$g_2$} (G);
  
  \tkzDrawSegments[>=stealth, arrowMeBig=>>>>](F,G A,H);
  \draw (F) -- node[below]{$h_2$} (G);
  \draw (A) -- node[left]{$h_2$} (H);
  
   \draw[double,thick, -stealth ,shorten >= 30pt, shorten <=30pt] (N) -- node[above right]{$\sigma(g_i,h_i)$}(M);
   \path (0,.5) coordinate (basepoint);
\end{tikzpicture}\label{eq:pictures}
\eeq
\end{center}
    \caption{A $\cG$-bundle over a surface $\Sigma$ determines an element of $\U(1)$ witnessing the relation $\prod_{i=1}^{\sf g} [g_i,h_i]\stackrel{\sigma}{\to} 1$.}
    \label{fig:4gons}
\end{figure}
\begin{rmk}
The categorified relation in Theorem~\ref{thm1} takes values in the \emph{torsor} $\U(1)$ rather than the group. 
A choice of fundamental cycle for $\Sigma$ identifies this $\U(1)$-torsor with the standard $\U(1)$, providing a numerical invariant. Theorem~\ref{thm1} follows from the fact that the Poincar\'e duality pairing is perfect. By virtue of coming from a pairing with the fundamental class, this construction can be viewed as a categorical avatar of transgression. 
\end{rmk}

This leads to the following entirely categorical interpretation of the Freed--Quinn line bundle.

\begin{cor}
    The fiber of the Freed--Quinn line bundle at a flat $G$-bundle $\pi_1\Sigma\to G$ consists of 2-commuting data $\sigma$ in Figure~\ref{fig:4gons}. 
    \end{cor}

\subsection{Acknowledgements} 
This paper began in conversations at the 2019 Mathematics Research Community, ``Geometric Representation Theory and Equivariant Elliptic Cohomology''. We thank the AMS and NSF for supporting this program. We warmly thank Eric Berry, Joseph Rennie, Apurva Nakade, and Emma Phillips, who were collaborators on the intial phase of this project. We also thank Matt Ando, Nora Ganter, Connor Grady, John Huerta, Eugene Lerman, Christopher Schommer-Pries, Nat Stapleton, and Stephan Stolz for stimulating conversations on these ideas. The second author thanks her students, Toby Caouette, Yves Lolelo, and Oc\'eane Perreault for their discussions and insights on this subject.

The first author was partially supported by the National Science Foundation under Grant Number DMS 2340239 and Grant Number DMS 2205835. The second author was partially supported by the Natural Sciences and Engineering Research Council of Canada under Discovery Grant RGPIN-2022-04104 and the Fonds de Recherche du Québec (Nature et Technologie) Research Support for New Academics under Grant 327706. The third author was partially supported by the National Science Foundation under Grant Number DMS 2316646.

\section{Background on (higher) principal bundles}

We begin with a review of ordinary $G$-bundles and 2-group bundles in a language convenient for our intended applications. 

\subsection{Groupoids and Lie groupoids}

A \emph{groupoid} is an essentially small category whose morphisms are all invertible. For a groupoid with objects $C_0$ and morphisms $C_1$, we denote the source and target maps by $s,t \colon C_1 \to C_0$ respectively, and the composition by $c \colon C_1 \times_{s,C_0,t} C_1=C_1 \times_{C_0} C_1 \to C_0$. We sometimes denote the groupoid itself as $\{C_1\rightrightarrows C_0\}$.

\begin{ex}[Action groupoids]\label{ex:actiongrpd}
A right $G$-action on a set $C$ determines the \emph{action groupoid} denoted $C\sq G$, whose objects are the set $C$ and morphisms $C\times G$. We adopt the convention that the source map is the action map and the target map is the projection; composition is determined by group multiplication in $G$.\footnote{This convention results in the cocycle condition for a \v Cech 1-cocycle $\rho$ reading as $\rho(y_1, y_2) \rho(y_2, y_3) = \rho(y_1, y_3)$, rather than $\rho(y_2, y_3)\rho(y_1, y_2) = \rho(y_1, y_3)$. Of course, other conventions work equally well with the corresponding reshuffling of indices.}
\end{ex}

A  \emph{Lie groupoid} is a groupoid object $\{C_1 \rightrightarrows C_0\}$ in smooth manifolds; in particular the source and target maps are required to be surjective submersions so that the fiber product $C_1 \times_{C_0} C_1$ is a smooth manifold. 

\begin{ex}[Smooth action groupoids] For a Lie group $G$ acting smoothly on a manifold $M$, the action groupoid $M \sq G$ from Example~\ref{ex:actiongrpd} is a Lie groupoid. 
\end{ex}
 
\begin{ex}[\v Cech groupoids]
For a surjective submersion $Y\to X$, let $Y^{[n]}$ denote the $n$-fold fibre product $Y \times_X \ldots \times_X Y$. We denote by 
$$
\check{C}(Y):=\{Y^{[2]}\rightrightarrows Y\}
$$
the \v{C}ech groupoid whose source, target, and composition maps are
$$
s(y_1,y_2)=y_2,\quad t(y_1,y_2)=y_1,\quad c(y_1,y_2,y_3)=(y_1,y_2)\circ(y_2,y_3)=(y_1,y_3)
$$
for $(y_1,y_2,y_3) \in Y^{[3]}$.

\end{ex}

\begin{ex}[Discrete groupoids and Lie groupoids]
There is a functor from groupoids to Lie groupoids that views the sets of objects and morphisms of a groupoid as 0-manifolds. The right adjoint to this functor takes the underlying sets corresponding to the manifold of objects and morphisms of a Lie groupoid. 
\end{ex}

A \emph{(smooth) functor} $f\colon X\to Y$ between Lie groupoids is a functor between their underlying categories that is a smooth map on objects $f_0\colon X_0\to Y_0$ and morphisms $f_1\colon X_1\to Y_1$. A \emph{(smooth) natural transformation} $\eta \colon f\Rightarrow g$ between functors is a natural transformation between underlying functors with the property that the map $\eta\colon X_0\to Y_1$ is smooth. 

Given a smooth functor $f\colon X\to Y$ between Lie groupoids, consider the diagrams of smooth manifolds 
\beq\label{eq:essentialequiv}
\begin{tikzpicture}[baseline=(basepoint)];
\node (A) at (0,0) {$Y_1\times_{Y_0}X_0$};
\node (B) at (3,0) {$X_0$};
\node (C) at (0,-1.5) {$Y_1$};
\node (D) at (3,-1.5) {$Y_0$,};
\draw[->] (A) to node [above] {$p_2$} (B);
\draw[->] (B) to node [right] {$f_0$} (D);
\draw[->] (A) to node [left] {$p_1$} (C);
\draw[->] (C) to node [below] {$s$} (D);
\path (0,-.75) coordinate (basepoint);
\end{tikzpicture}\qquad 
\begin{tikzpicture}[baseline=(basepoint)];
\node (A) at (0,0) {$X_1$};
\node (B) at (3,0) {$Y_1$};
\node (C) at (0,-1.5) {$X_0\times X_0$};
\node (D) at (3,-1.5) {$Y_0\times Y_0.$};
\draw[->] (A) to node [above] {$f_1$} (B);
\draw[->] (B) to node [right] {$t\times s$} (D);
\draw[->] (A) to node [left] {$t\times s$} (C);
\draw[->] (C) to node [below] {$f_0\times f_0$} (D);
\path (0,-.75) coordinate (basepoint);
\end{tikzpicture}
\eeq

\begin{defn}[{e.g., \cite[Definition 3.5]{Lerman}}] \label{defn:localequiv}
A smooth functor $f\colon X\to Y$ is an \emph{essential equivalence} if $Y_1\times_{Y_0}X_0\to Y_0$ on the left in~\eqref{eq:essentialequiv} is a surjective submersion, and the diagram on the right in~\eqref{eq:essentialequiv} is a pullback. 
\end{defn}

When $X$ and $Y$ are discrete, the diagrams~\eqref{eq:essentialequiv} correspond to $f$ being essentially surjective and fully faithful.

\subsection{Covers and presentations of groupoids}\label{sec:groupoidprelim}
Let $\sC$ be a groupoid. 
For a set $C_0$ viewed as a discrete category, a \emph{cover} of $\sC$ is an epimorphism $\cc\colon C_0\twoheadrightarrow \sC$. A choice of cover determines a groupoid denoted $\cc^*\sC$ with objects $C_0$ and morphisms the (weak) pullback $C_1=C_0\times_\sC C_0$, whose fiber over a pair of points $(x,y)\in C_0\times C_0$ is the set of morphisms in $\sC$ from $\cc(y)$ to $\cc(x)$ (and in particular, $C_1$ is a set). The source, target, unit and composition come from canonical maps between fiber products,
\beq\label{eq:presentationforcover}
&&s=p_2,t=p_1\colon C_0\times_\sC C_0\to C_0, \quad e=\Delta\colon C_0\to C_0\times_\sC C_0, \quad c=p_{13}\colon C_0\times_\sC C_0\times_\sC C_0\to C_0\times_\sC C_0
\eeq

By construction, there is a canonical equivalence $\cc^*\sC\to \sC$. We refer to $\cc^*\sC$ as the \emph{presentation} of $\sC$ determined by the cover $\cc \colon C_0 \twoheadrightarrow \sC$. 

\begin{defn} \label{defn:weakaction}
For a groupoid $\sC$ and group $\Gamma$, a (right) \emph{$\Gamma$-action on $\sC$} is a functor $\act\colon \sC\times \Gamma\to \sC$ together with natural isomorphisms of functors witnessing 2-commutativity of the diagrams
\beq\label{eq:weakactiondata}
\begin{tikzpicture}[baseline=(basepoint)];
\node (A) at (0,0) {$\sC\times \Gamma\times \Gamma$};
\node (B) at (4,0) {$\sC\times \Gamma$}; 
\node (C) at (0,-1.2) {$\sC\times \Gamma$};
\node (D) at (4,-1.2) {$\sC$};
\draw[->] (A) to node [above] {$\id_{\sC}\times m$} (B);
\draw[->] (A) to node [left] {$\act\times \id_\Gamma$} (C);
\draw[->] (C) to node [above] {$\act$} (D);
\draw[->] (B) to node [right] {$\act$}  (D);
\path (0,-.75) coordinate (basepoint);
\end{tikzpicture}\qquad 
\begin{tikzpicture}[baseline=(basepoint)];
\node (A) at (0,0) {$\sC\times \Gamma$};
\node (B) at (2,0) {$\sC$,};
\node (C) at (1.25,-1.2) {$\sC$};
\draw[->] (A) to node [above] {$\act$} (B);
\draw[->] (C) to node [left] {$\id_\sC\times e$} (A);
\draw[->] (C) to node [right] {$\id$} (B);
\path (0,-.75) coordinate (basepoint);
\end{tikzpicture}
\eeq
where $m$ is the multiplication on $\Gamma$, and $e\colon \pt\to \Gamma$ is the inclusion of the identity element. These data are required to satisfy associator and unitor axioms. A $\Gamma$-action on $\sC$ is \emph{strict} if the natural isomorphisms in~\eqref{eq:weakactiondata} are identities. 
\end{defn}

\begin{defn}
For a group $\Gamma$ acting on a groupoid $\sC$, the \emph{quotient groupoid}, denoted $\sC\sq \Gamma$, has the same objects as $\sC$ and morphisms given by pairs $(f,g)$ for $g\in \Gamma$ and $f\colon y \to x\cdot g$ a morphism in $\sC$. Composition is inherited from composition in $\sC$, group multiplication in $\Gamma$, and the 2-commuting data in the definition of a $\Gamma$-action. 
\end{defn}

\begin{defn}
An \emph{equivariant functor} between two categories $\sC_1, \sC_2$ with $\Gamma$-action is a functor $F \colon \sC_1 \to \sC_2$ together with a natural isomorphism witnessing equivariance. 
\end{defn}

\begin{rmk}
Given a strict $\Gamma$-action on $\sC$ and an equivalence $\sC\to \sC'$, the induced $\Gamma$-action on $\sC'$ typically fails to be strict. With this in mind, given a weak $\Gamma$-action on $\sC$, we can look for a cover of $\sC$ such that the associated presentation carries a strict $\Gamma$-action. 
\end{rmk}

 \begin{prop}\label{prop: equivariant cover}
 Let $\sC$ be a groupoid equipped with an action of a group $\Gamma$. Let $\cc\colon C_0\to \sC$ be an equivariant cover: i.e. we assume that $\Gamma$ acts on the set $C_0$ and the epimorphism $\cc$ is equivariant. Then the associated presentation $\cc^*\sC$ of $\sC$ carries a strict $\Gamma$-action. 
 \end{prop}
 \bp
 The objects of $\cc^*\sC$ are given by the set $C_0$ with the given $\Gamma$-action. The $\Gamma$-action on morphisms is given by
 \begin{align*}
     (x,y, f \colon \cc(y) \to \cc(x)) \cdot g = (xg, yg, \widetilde{f}), 
 \end{align*}
 where $\widetilde{f}$ is the composition of $fg$ with equivariance isomorphisms:
 \begin{align*}
     \cc(yg) \EquivTo \cc(y)g \xrightarrow{fg} \cc(x)g \EquivTo \cc(xg).
 \end{align*}
 It is easy to check that this gives a strict $\Gamma$-action on $\cc^*\sC$ such that the natural equivalence $\cc^*\sC \to \sC$ is strictly equivariant. 
 \ep

\begin{ex}\label{eg:coprodofgrps}
For any groupoid $\sC$, a choice of one object $x$ in each isomorphism class gives an epimorphism $
\pi_0\sC\twoheadrightarrow \sC$. The associated presentation of $\sC$ is then $\coprod_{[x]\in \pi_0(\sC)} \pt\sq \Aut_\sC(x)$. However, since the choice of representatives $x$ is in general not compatible with the $\Gamma$-action, it is more difficult to describe the $\Gamma$-action on this presentation, and it is not strict in general.
\end{ex}

One can define presentations for Lie groupoids analogously to~\eqref{eq:presentationforcover} and group actions on Lie groupoids analogously to Definition~\ref{defn:weakaction}; we do not require these below so we omit these definitions. 

\subsection{Presentations of $\Bun_G(X)$}
 Let $G$ be a discrete group and $X$ a smooth manifold. We will consider the presentations of $\Bun_G(X)$ determined by three different equivariant covers of $\Bun_G(X)$; each presentation has different advantages. The presentation induced from the cover by \v Cech cocycles will generalize most naturally to our definition of principal $2$-group bundles. The triangulation cover compares most closely with the work of Freed--Quinn. Finally, the holonomy cover will be more conducive to calculations. The ideas of this section are standard; we recall them here to give a uniform description and establish our conventions. 

\begin{ex}[The \v{C}ech cocycle presentation] \label{ex:Cech}
 Let ${\check{C}}^1(X,G)$ denote the collection of pairs $(u,\rho)$ with $u\colon Y\to X$ a surjective submersion and $\rho\colon Y\times_{X} Y\to G$ a $G$-valued 1-cocycle. There is an epimorphism
    \beq\label{eq:epi2}
{\check{C}}^1(X,G)\twoheadrightarrow \Bun_G(X), \quad (u,\rho)\mapsto P_{u,\rho} =(Y \times G) /{\sim_\rho}
    \eeq
where $(y_1, \rho(y_1, y_2)g) \sim_\rho (y_2,g)$ for $(y_1,y_2) \in Y \times_X Y, g \in G$. As every $G$-bundle trivializes on some cover,~\eqref{eq:epi2} is an epimorphism. This gives the presentation
\beq\label{eq:Bungpresent1}
\Bun_G(X)\simeq ({\check{C}}^1(X,G)\times_{\Bun_G(X)} {\check{C}}^1(X,G)\rightrightarrows \check{C}^1(X,G)),
\eeq
which we will refer to as the \emph{\v Cech cocycle presentation}. More precisely, a morphism in this presentation is given by a tuple $$((u_2 \colon Y_2 \to X, \rho_2),(u_1 \colon Y_1 \to X, \rho_1), \varphi \colon P_{u_1, \rho_1} \EquivTo P_{u_2, \rho_2})$$ consisting of two cocycles and an isomorphism between the induced principal bundles. Recall that for any mutual refinement $v= u_1 \circ v_1 = u_2 \circ v_2 \colon Z \to X$ of the covers $u_i \colon Y_1 \to X$ as in 
\beq\label{eq: refinement}
\begin{tikzpicture}[baseline=(basepoint)];
\node (A) at (0,0) {$Z$};
\node (B) at (1.2,0) {$Y_1$}; 
\node (C) at (0,-1.2) {$Y_2$};
\node (D) at (1.2,-1.2) {$X,$};
\draw[->] (A) to node [above] {$v_1$} (B);
\draw[->] (A) to node [left] {$v_2$} (C);
\draw[->] (C) to node [below] {$u_2$} (D);
\draw[->] (B) to node [right] {$u_1$} (D);
\path (0,-.75) coordinate (basepoint);
\end{tikzpicture}
\eeq
and under the canonical equivalences $Z \times G \EquivTo v^*P_{u_i, \rho_i}$ the pullback $v^*\varphi \colon Z \times G \EquivTo Z \times G$ is of the form $(z,g) \mapsto (z, h(z)g)$, for $h \colon Z \to G$ satisfying $v_2^*\rho_2\cdot p_2^* h=p_1^* h\cdot v_1^*\rho_1$. Conversely, given such a mutual refinement $u_1 \circ v_1 = u_2 \circ v_1 \colon Z \to G$ and a function $h\colon Z \to G$ satisfying this condition, we obtain an isomorphism $P_{u_1, \rho_1} \EquivTo P_{u_2, \rho_2}$, which we will denote by $\varphi_{Z, h}$ (or $\varphi_{Z, v_1, v_2, h}$ if the maps $v_i$ are not clear from context). However, two such pairs $(Z, h), (Z', h')$ induce the same isomorphism in $\Bun_G(X)$ if there is a common refinement of $Z$ and $Z'$ over which $h$ and $h'$ pull back to the same function. That is, the morphisms in the presentation of $\Bun_G(X)$ given by \eqref{eq:Bungpresent1} are parametrized by equivalence classes 
\beq\label{eq: cech morphisms}
\left\{\varphi_{Z,h} = [
Y_1 \xleftarrow{v_1} Z \xrightarrow{v_2} Y_2, \ h \colon Z\to G ] \mid \begin{array}{c} u_1 \circ v_1 = u_2 \circ v_2\ {\rm surj.\ sub.} \\  v_2^*\rho_2\cdot p_2^* h=p_1^* h\cdot v_1^*\rho_1\end{array} \right\}
\eeq
where $\varphi_{Z,h}=\varphi_{Z',h'}\colon (u_1, \rho_1) \to (u_2, \rho_2)$ whenever there exists a refinement of $Z$ and $Z'$ on which $h$ and $h'$ become equal. 

\begin{ex}
\label{rmk: pullback gives morphism}
    We observe that refinements of covers are themselves morphisms of $G$-bundles in the \v Cech cocycle presentation: given $(u \colon Y \to X, \rho) \in \check{C}^1(X,G)$ and a refinement $v \colon Z \to Y$, we obtain a morphism $\varphi_{Z, \id_Z, v, h=1} \colon (u, \rho) \to (u \circ v, v^*\rho)$.  
\end{ex}

Finally, we define the action of the group $\Diff(X)$ on $\check{C}^1(X, G)$,
$$
{\check{C}}^1(X,G)\times \Diff(X)\to {\check{C}}^1(X,G),\qquad (u, \rho)\cdot f = (f^{-1} \circ u, \rho).
$$ 
\end{ex}
\begin{prop}\label{prop: cech equivariance}
The epimorphism $\check{C}^1(X,G) \twoheadrightarrow\Bun_G(X)$ is $\Diff(X)$-equivariant, and hence the \v Cech cocycle presentation of $\Bun_G(X)$ carries a strict $\Diff(X)$-action. 
\end{prop}

\begin{proof}
Equivariance is additional data, namely a compatible family consisting of, for each $f \in \Diff(X)$ and for each $(u,\rho)\in \check{C}^1(X,G)$, an isomorphism of $G$-bundles $P_{f^{-1} \circ u, \rho} \EquivTo f^*(P_{u, \rho})$, as in the following diagram: 
\beq \label{eq: equivariance pullback diagram}
\begin{tikzpicture}[baseline=(basepoint)];
\node (F) at (-2,.5) {$P_{f^{-1}\circ u,\rho}$};
\node (A) at (0,0) {$f^*(P_{u,\rho})$};
\node (B) at (2,0) {$P_{u,\rho}$}; 
\node (C) at (0,-1.5) {$X$};
\node (D) at (2,-1.5) {$X.$};
\node (E) at (1,-.5) {$\lrcorner$};
\draw[->] (A) to (B);
\draw[->] (A) to (C);
\draw[->] (C) to node [below] {$f$} (D);
\draw[->] (B) to (D);
\draw[->,bend right=20] (F) to (C);
\draw[->,bend left=20] (F) to (B);
\draw[->,dashed] (F) to (A);
\path (0,-.75) coordinate (basepoint);
\end{tikzpicture}
\eeq
The morphism on the left must be the structure map of the bundle $P_{f^{-1} \circ u, \rho}$, so to construct this dashed arrow it suffices to provide the top outer map, which should be a $G$-equivariant map making the outer diagram commute. The resulting dashed arrow will then be a morphism of $G$-bundles, and hence automatically an isomorphism.

We note that the fiber product $Y \times_X Y$ appearing in the equivalence relations defining $P_{u, \rho}$ and $P_{f^{-1}\circ u, \rho}$ in Equation \eqref{eq:epi2} is the same space whether the maps $Y \to X$ are given by $f^{-1}\circ u$ or $u$, differing only in whether the map $Y\times_X Y \to X$ corresponds to $(y_1, y_2) \mapsto f^{-1} \circ u(y_1)$ or $(y_1, y_2) \mapsto u(y_1)$. Hence, $P_{f^{-1} \circ u, \rho}$ and $P_{u, \rho}$ can be canonically identified as spaces with $G$-action. This identification is the desired top outer map. 

Because they are all defined using the universal property of pullbacks, the resulting isomorphisms $P_{f^{-1} \circ u, \rho} \EquivTo f^*(P_{u, \rho})$ are compatible with multiplication in $\Diff(X)$: that is, the composition
\begin{align*}
    P_{f_2^{-1}\circ f_1^{-1} \circ u, \rho} \EquivTo f_2^*(P_{f_1^{-1}\circ u, \rho}) \EquivTo f_2^*f_1^*(P_{u, \rho}) \EquivTo (f_1 \circ f_2)^*(P_{u, \rho})
\end{align*}
agrees with the isomorphism $P_{(f_1 \circ f_2)^{-1}, \rho} \EquivTo (f_1 \circ f_2)^*(P_{u, \rho}) $ defined by the version of the diagram \eqref{eq: equivariance pullback diagram} with $f = f_1 \circ f_2$.
\end{proof}

\begin{ex}[The triangulation presentation]\label{ex: triangulation}
Now suppose $X=\Sigma$ is a surface. Let $\Delta(\Sigma,G)$ denote the collection of pairs $(\tri,\rho)$ where $\tri$ is a triangulation of $\Sigma$, $u_\tri\colon Y_\tri\to \Sigma$ is the open cover coming from taking the star of each vertex of the triangulation $\tri$, and $\rho$ is a cocycle $\rho\colon Y_\tri\times_\Sigma Y_\tri\to G$. Using that these triangulation covers are good, specializing~\eqref{eq:epi2} we have the epimorphism
    \beq\label{eq:epi3}
\Delta(\Sigma,G)\twoheadrightarrow \Bun_G(\Sigma), \quad (\tri,\rho)\mapsto P_{u_\tri,\rho}. 
    \eeq
Let us describe the morphisms in the associated presentation of $\Bun_G(\Sigma)$, which we will refer to as the \emph{triangulation presentation}. The description is similar to that of the morphisms in the \v Cech cocycle presentation, but under the equivalence relation on morphisms in ~\eqref{eq: cech morphisms}, we can assume that the mutual refinement of covers corresponds to a mutual refinement $T$ of the underlying triangulations $\tri_1$ and $\tri_2$, giving giving $v_{T,i} \colon Y_T \to Y_{\tri_i}, i = 1,2$. Hence, we have
$$
\Delta(\Sigma,G)\times_{\Bun_G(\Sigma)}\Delta(\Sigma,G)\simeq \{ [T, h\colon Z_T\to G]  \mid (v_T)_2^*\rho_2\cdot p_2^*h=p_1^*h\cdot (v_T)_1^*\rho_1\},
$$
where $[T, h] = [T', h']$ if there exists a mutual refinement $T''$ of the triangulations $T$ and $T'$ such that $h$ and $h'$ are equal when pulled back to the associated open cover $Y_{T''}$.

There is a natural action of $\Diff(\Sigma)$ on $\Delta(\Sigma, G)$, defined as follows. Given a triangulation $\tri$ and a diffeomorphism $f\colon \Sigma\to \Sigma$, we obtain a new triangulation $f^{-1}(\tri)$ (where the inverse is to give a right action by $\Diff(\Sigma)$). The map $f$ itself determines a diffeomorphism of open covers $f_\tri\colon Y_{f^{-1}(\tri)}\to Y_{\tri}$ covering~$f$. Then we define 
\beq\label{Eq:triangulationDiff}
&&\Delta(\Sigma,G)\times \Diff(\Sigma)\to \Delta(\Sigma,G),\qquad     (\tri, \rho) \cdot f = (f^{-1}(\tri), \rho \circ (f_\tri \times f_\tri))=(f^{-1}(\tri), f_\tri^*\rho).
\eeq
\end{ex}

\begin{prop}\label{prop: triangulation equivariance}
The epimorphism $\Delta(\Sigma, G) \twoheadrightarrow\Bun_G(\Sigma)$ is $\Diff(\Sigma)$-equivariant, and hence the triangulation presentation of $\Bun_G(\Sigma)$ carries a strict $\Diff(\Sigma)$-action. 
\end{prop}

\begin{proof}
Analogously to the proof of Proposition \ref{prop: cech equivariance}, we need to provide a compatible family of isomorphisms $P_{u_{f^{-1}(\tri)}, f_\tri^*\rho} \EquivTo f^*(P_{u_\tri, \rho})$ indexed by $(f \in \Diff(\Sigma), (\tri, \rho) \in \Delta(\Sigma, G))$. By Example \ref{rmk: pullback gives morphism}, $(f^{-1} \circ u_{\tri}, \rho)$ is canonically isomorphic to its pullback along the refinement $f_\tri \colon Y_{f^{-1}(\tri)} \to Y_\tri$:
\begin{align*}
    (f^{-1} \circ u_\tri, \rho) \cong (f^{-1} \circ u_\tri \circ f_{\tri}, f_\tri^*\rho)= (u_{f^{-1}(\tri)}, f_\tri^* \rho ).
\end{align*}
We compose the resulting isomorphism of $G$-bundles with the equivariance isomorphism from the proof of Proposition \ref{prop: cech equivariance}, to obtain
\begin{align*}
   P_{u_{f^{-1}(\tri)}, f_\tri^*\rho} \EquivTo P_{f^{-1} \circ u_\tri, \rho} \EquivTo f^*(P_{u_\tri, \rho}), 
\end{align*}
as desired. It is straightforward to check the compatibility of these morphisms with composition of diffeomorphisms. 
\end{proof}

\begin{ex}[The holonomy presentation] \label{ex: hol pres}
Suppose $X$ is a connected manifold with universal cover $\tilde{u} \colon \tilde{X} \to X$. As $G$-bundles are particular instances of covering spaces, all $G$-bundles trivialize when pulled back to~$\tilde{X}$. Hence \eqref{eq:epi2} remains an epimorphism when restricted to pairs $(\tilde{u}, \rho)$. Furthermore, we have an isomorphism 
\beq\label{eq:fiberedproductforuniv}
\tilde{X}\times \pi_1X \cong \tilde{X}\times_X\tilde{X}, \qquad (y, a) \mapsto (y, y\cdot a),
\eeq
which induces a bijection between homomorphisms $\rhohol \colon \pi_1 X \to G$ and (locally constant) cocycles $\rho \colon \tilde{X} \times_X \tilde{X} \to G$ via the formula $\rho(y, y\cdot a) = \rhohol(a), y \in \tilde{X}, a \in \pi_1X$. This provides an embedding 
\begin{align}\label{eq: holonomy-cech embedding}
\iota \colon \Hom(\pi_1X,G)\hookrightarrow \check{C}^1(X,G),
\end{align}
and induces from ~\eqref{eq:epi2} an epimorphism
\begin{align}\label{eq:epihom}
    \Hom(\pi_1X,G)\twoheadrightarrow \Bun_G(X),\qquad \rhohol\mapsto P_{\rhohol} = P_{\tilde{u}, \rho},
\end{align}
for $(\tilde u,\rho) = \iota ({\rhohol})$. We note that the equivalence relation $\sim_{\rho}$ on $\tilde{X} \times G$ used in defining $P_{\tilde{u}, \rho}$ \eqref{eq:epi2} can be formulated in terms of ${\rhohol}$: $(y\cdot a, g) \sim_\rho (y, {\rhohol}(a)g)$ for $y \in \tilde{X}, a \in \pi_1 X, g \in G$.  

Let us compute the fiber product of~\eqref{eq:epihom} over itself, in order to describe the morphisms in the associated presentation of $\Bun_G(\Sigma)$, the \emph{holonomy presentation}. We have
\beq\label{eq:BunGfibprod}
\Hom(\pi_1X,G)\times_{\Bun_G(X)} \Hom(\pi_1X,G)\simeq \Hom(\pi_1X,G)\times G,
\eeq
using that an isomorphism $\varphi\colon P_{\rhohol}\to P_{\rhohol'}$ of $G$-bundles is determined by $g\in G$ with $g\rhohol g^{-1}=\rhohol'$, e.g., by viewing $g\in G$ as an automorphism of the trivial $G$-bundle $\tilde{X}\times G$. This affords the equivalence
\beq\label{eq:homgrpddescription}
\iota \colon  \Hom(\pi_1X,G)\sq G \to \Bun_G(X)
\eeq
for the conjugation $G$-action on homomorphisms. In turn, we can identify the action groupoid $\Hom(\pi_1 X, G)\sq G$ with the groupoid of functors $\Fun(* \sq \pi_1X, * \sq G)$.

The action of $\Diff(X)$ on $\Hom(\pi_1X,G)$ is given by precomposing with the isomorphism $f_*\colon \pi_1 X\to \pi_1X$ associated to a diffeomorphism $f\colon X\to X$. As $\pi_1X$ is discrete, this action factors through $\pi_0\Diff(X)$, the mapping class group. 
\end{ex}

We caution that the inclusion of sets $\Hom(\pi_1X,G)\hookrightarrow \check{C}^1(X,G)$ is not $\Diff(X)$-equivariant; however, it will follow from Proposition \ref{prop: cech equivariance} above and Proposition \ref{prop: holonomy equivariance} that the two actions of $\Diff(X)$ agree up to isomorphism in $\Bun_G(X)$. 

\begin{prop}\label{prop: holonomy equivariance}
The epimorphism $\Hom(\pi_1X, G) \twoheadrightarrow \Bun_G(X)$ is $\Diff(X)$-equivariant, and hence the holonomy presentation of $\Bun_G(X)$ carries a strict $\Diff(X)$-action. 
\end{prop}

\begin{proof}
    Analogously to the proofs of Propositions \ref{prop: cech equivariance} and \ref{prop: triangulation equivariance}, we need to provide a compatible family of isomorphisms $P_{{\rhohol}\circ f_*} \EquivTo f^* P_{\rhohol}$ indexed by $(f \in \Diff(X), {\rhohol} \in \Hom(\pi_1X, G))$. Choose a lift $\tilde f \colon \tilde X \to \tilde X$ of $f$, and view it as a refinement of the cover $\tilde X \to X$. Let $\iota({\rhohol}) = (\tilde u, \rho)$, and pull back the cocycle $(f^{-1} \circ \tilde u, {\rho})$ along the refinement $\tilde f$. As in Example \ref{rmk: pullback gives morphism}, we obtain an isomorphism
    \begin{align*}
        (f^{-1} \circ \tilde u, \rho) \cong (f^{-1}\circ \tilde u \circ \tilde f, \tilde f^* \rho) .
    \end{align*}
    We calculate that $f^{-1}\circ \tilde u \circ \tilde f = u$, and that $\tilde f^*\rho$ acts on a point $(y , y\cdot a) \in \tilde X \times_X \tilde X$ (with $y \in \tilde X, a \in \pi_1 X$ by
    \begin{align*}
        \tilde f^* \rho(y , y\cdot a) = \rho (\tilde f(y ), \tilde f(y\cdot a)) = \rho(\tilde f(y) , \tilde f(y)\cdot f_*(a)) = {\rhohol} \circ f_* (a). 
    \end{align*}
    In other words, we have a natural isomorphism $P_{\rhohol \circ f_*(a)} \cong P_{f^{-1} \circ \tilde u, \rho}$, and we compose it with the morphism $P_{f^{-1} \circ \tilde u, \rho} \to f^*(P_{\tilde u, \rho)}=f^*(P_{\rhohol})$ from the proof of Proposition \ref{prop: cech equivariance} to get the desired isomorphism. Again, it is straightforward to check the compatibility of these morphisms with composition of diffeomorphisms. 
\end{proof}

\begin{rmk}[A skeletal presentation] 
We remark that as in Example \ref{eg:coprodofgrps}, a choice of representative of each conjugacy class of homomorphisms $\pi_1 X\to G$ gives a presentation 
$$
\Hom(\pi_1X,G)\sq G\simeq \coprod_{[\rhohol]} \pt\sq \Stab_G(\rhohol)
$$
for the coproduct indexed by the chosen representatives of conjugacy classes, and $\Stab_G(\rhohol)<G$ the stabilizer of $\rhohol$. This presentation is convenient for many applications and is often used throughout the literature, but it is not suitable for our purposes because it is not compatible with the action of $\Diff(X)$.
\end{rmk}

\subsection{2-groups and 2-group bundles}\label{sec:2grp}

\begin{defn} \label{defn:2group}
A (discrete) \emph{2-group} is a monoidal groupoid $(\cG,\otimes,\one)$ where every object is (weakly) $\otimes$-invertible, meaning for every object $x$ there exists an object $x^{-1}$ and isomorphisms $x\otimes x^{-1}\simeq \one\simeq x^{-1}\otimes x.$ A \emph{1-homomorphism} between 2-groups is a lax monoidal functor.
A \emph{2-homomorphism} between 1-homomorphisms is a lax monoidal transformation.
\end{defn}

The collection of 2-groups, 1-homomorphisms, and 2-homomorphisms has the structure of a bicategory. 

\begin{ex}\label{ex:ordinarygroup}
An ordinary group $G$ determines a 2-group whose underlying groupoid has only identity morphisms, with monoidal structure given by the group multiplication on~$G$. A homomorphism of groups is equivalent data to a monoidal functor between the corresponding monoidal categories. Hence, the category of groups and homomorphisms admits a faithful embedding into the bicategory of 2-groups. 
\end{ex}

\begin{ex}\label{ex:ptmodA}
When $G=A$ is abelian, the action groupoid $\pt\sq A$ for the trivial action on $\{*\}$ has a single object and morphisms $A$. It has the structure of a 2-group with monoidal structure determined by multiplication in $A$.
\end{ex}

In her thesis \cite{sinh}, Ho\`ang gives the following classification of 2-groups as a combination of these examples. 
\begin{prop}\label{prop:2grpconstr}
A 2-group $\cG$ is determined up to 1-isomorphism by
\begin{enumerate}
\item a group $G=\pi_0(\cG)$, the set of isomorphism classes of objects in $\cG$ with group structure inherited from the monoidal structure on~$\cG$; 
    \item an abelian group $A=\Aut(1_G)$;
    \item a $G$-action on $A$, $G=\pi_0(\cG) \to \Aut(A)$;
    \item a class $[\alpha]\in \rmH^3(G;A)$.
\end{enumerate}
Note that this can be described in the form of a sequence of 1-homomorphisms~\eqref{eq:catextension}
\beq\label{eq:catextension}
1\to \pt\sq A\xrightarrow{\iota} \cG\xrightarrow{\pi_0} G\to 1,
\eeq
where $G$ is regarded as a 2-group, see Example~\ref{ex:ordinarygroup}. 
\end{prop} 
\begin{rmk}

The above classification also appears as \cite[Theorem 43]{BL04}. For categorical central extensions (where the $G$-action on $A$ is trivial) it also follows from \cite[Theorem 99]{SP11}. 

\end{rmk}

Hereafter, we will assume $A$ is a trivial $G$-module, i.e., the extension~\eqref{eq:catextension} is central. As any class $[\alpha]\in \rmH^3(G;A)$ can be represented by a normalized 3-cocycle, 
$$
\alpha(1_G,g_1,g_2)=\alpha(g_1,1_G,g_2)=\alpha(g_1,g_2,1_G)=1_A, \text{ for all } g_1, g_2\in G,
$$
we shall furthermore assume that $\alpha$ is normalized. 

\begin{convention}\label{notation:GAalpha}
From now on, we let $\cG=\cG(G, A, \alpha)$ be the 2-group associated to a discrete group $G$, trivial $G$-module $A$, and a normalized $A$-valued 3-cocycle~$\alpha$. \end{convention}

\begin{ex}\label{ex: alpha}
For $N\in \N$, the inclusion $\mu_N\colon \Z/N\Z\hookrightarrow S^1$ as the $N$th roots of unity determines 2-group extensions of $\Z/N$ classified by the image of the generator,
\beq\label{eq:CPinfinity}
\Z\simeq \rmH^4(\CP^\infty;\Z) \simeq \rmH^4(B\U(1);\Z)\xrightarrow{\mu_N^*}\rmH^4(B\Z/N\Z;\Z) \simeq \rmH^3(\Z/N\Z;\U(1)).
\eeq
Below we use the explicit 3-cocycle $\alpha\colon \Z/N\Z\times \Z/N\Z\times \Z/N\Z\to \U(1)$ representative of this generator given by\footnote{The second author thanks her student Toby Caouette, who taught her this convenient formula for a representative of the generator during an undergraduate research project.}
\begin{align}\label{eq:Znalpha}
    \alpha(j + N\Z, k + N\Z, l + N\Z) = \left\{ \begin{array}{ll}
       1  & \text{ if }k + l < N \\
       e^{\frac{2j\pi i}{N} }  & \text{ if } k + l \ge N,
    \end{array}\right.  
\end{align}
for representatives $j,k,l \in \{0, 1, \ldots, N-1\}$.
\end{ex}

\begin{ex}
Given a character $\varphi \colon G\to \U(1)$, we obtain a 2-group classified by the pullback of the generator~\eqref{eq:Znalpha} along~$\varphi$. As $\varphi$ factors through a finite subgroup of $G$, in fact this pulls back from an extension $\cG(\Z/|G|\Z,\U(1),\alpha)$ classified by~\eqref{eq:Znalpha}. 
\end{ex}

Just as there are a variety of presentations of $G$-bundles, there are several ways to construct the 2-groupoid of $\cG$-bundles $\Bun_\cG(X)$ for $\cG$ a categorical central extention~\eqref{eq:catextension}. We start from the definition that is most concrete, namely the generalization of the \v{C}ech description of $\Bun_G(X)$ from Example~\ref{ex:Cech}. However, one can show that this is equivalent to various other definitions, including as zig-zags of functors between $X$ and $* \sq \cG$ (Remark \ref{rmk:bundle as functor}), as stacks over $X$ with $\cG$-action (Remark \ref{rmk:otherdefs}), or in terms of weak representations, providing higher holonomy data (Proposition \ref{prop:weakrepn}).

\begin{defn}\label{defn:prinbund1}
Let $\calg=\calg(G, A, \alpha)$ and let $X$ be a manifold. 
    The bicategory $\Bun_\cG(X)$ of flat $\cG$-bundles on $X$ consists of:
    \begin{itemize}
        \item \textbf{objects:} $(u, \rho, \gamma)$, where $u\colon Y\to X$ is a surjective submersion; 
        $\rho\colon Y^{[2]}\to G$ is a locally constant map satisfying the (ordinary) cocycle condition
            \beq \label{eq: condition for rho}
            p_{13}^*\rho = p_{12}^*\rho\cdot p_{23}^*\rho \colon Y^{[3]}\to G
            \eeq
            for $p_{12},p_{23}, p_{13}\colon Y^{[3]}\to Y^{[2]}$ the projections and the composition; 
        and $\gamma\colon Y^{[3]}\to A$ is a locally constant map satisfying the conditions
            \beq\label{eq:highercompatcond2}
            &&\rho^*\alpha=d \gamma\colon Y^{[4]}\to A,\qquad \gamma(y_1, y_2, y_2)=\gamma(y_2,y_2, y_3)=1_A, \label{eq: condition for gamma}
            \eeq
            for all $(y_1, y_2, y_3)\in Y^{[3]}$, where $d$ is the \v{C}ech differential on $A$-valued cochains.

        \item \textbf{1-morphisms:} $(u_1, \rho_1, \gamma_1) \to (u_2, \rho_2, \gamma_2) $ is given by data $(Z, v_1, v_2, h, \eta)$, where $v_i \colon Z \to Y_i, i = 1,2$ are smooth maps such that $u_1 \circ v_1 = u_2 \circ v_2$ are surjective submersions;
        $h\colon Z\to G$ is a locally constant map satisfying 
            \beq\label{eq: h condition}
            v_2^*\rho_2 \cdot p_2^* h =  p_1^* h \cdot v_1^*\rho_1 \colon Z^{[2]} \to G; 
            \eeq 
            for $p_1, p_2\colon Z^{[2]}\to Z$ the target and source maps;
        and $\eta\colon Z^{[2]}\to A$ is a locally constant map satisfying
        \beq\label{eq:eta condition}
         && d\eta (z_1,z_2, z_3) = \frac{v_1^*\gamma_1(z_1,z_2,z_3)}{v_2^*\gamma_2(z_1,z_2,z_3)}\cdot\frac{\alpha(v_2^*\rho_2(z_1,z_2), h(z_2),v_1^*\rho_1(z_2,z_3))}{\alpha(h(z_1),v_1^*\rho_1(z_1,z_2), v_1^*\rho_1(z_2,z_3)) \alpha(v_2^*\rho_2(z_1,z_2), v_2^*\rho_2(z_2,z_3), h(z_3))},
        \eeq
        for all $(z_1,z_2,z_3)\in Z^{[3]}$.

        \item \textbf{2-morphisms:}  Let $(Z_j, v_{1j}, v_{2j}, h_j, \eta_j), j = 3,4,$ be two 1-morphisms between objects $(u_1, \rho_1, \gamma_1)$ and $(u_2, \rho_2, \gamma_2)$. A 2-morphism is given by data $(Z, v_3, v_4, \omega)$, where 
        $v_j \colon Z \to Z_j$ are essential equivalences such that
        $ v_{i3} \circ v_3 = v_{i4} \circ v_4$, $i=1,2$;
        and $\omega\colon Z\to A$ is a locally constant map satisfying 
            \beq\label{eq: omega condition}
            \frac{\omega(z_2)}{\omega(z_1)} = \frac{v_3^*\eta_3(z_1, z_2)}{v_4^*\eta_4(z_1, z_2)}
        \eeq 
        for $(z_1, z_2) \in Z^{[2]}$.
        Note that such a 2-morphism exists only if $v_3^*h_3=v_4^*h_4$.
        
        Two sets of data of the form $(Z, v_3, v_4, \omega)$ represent the same 2-morphism if they agree ``upon refinement,'' that is after being pulled back along compatible essential equivalences.  

    \end{itemize}
Note that all 1-morphisms and 2-morphisms are invertible, i.e., $\Bun_\cG(X)$ is a 2-groupoid.
\end{defn}

\begin{notation}\label{notation:calp}
To make better contact with the geometric ideas, below we use the notation $\calp_{u,\rho,\gamma}$ to denote the object $(u,\rho, \gamma)\in \Bun_{\cG}(X)$, or simply $\calp\in \Bun_{\cG}(X)$ when the data $(u,\rho,\gamma)$ is clear in context. Similarly, we use the notation $\iso_{Z,h,\eta}$ or simply $\varphi$ for the 1-morphism $(Z, v_1, v_2, h,\eta)$ in $\Bun_\cG(X)$. Finally, we use the shorthand notation $\omega\colon \iso\Rightarrow\iso'$ for a 2-morphism $(Z, v_3, v_4, \omega)$ in $\Bun_\cG(X)$. We note that for every $\cG$-bundle $\cP_{u,\rho,\gamma}$, we have an underlying $G$-bundle $P_{u,\rho}$.

\end{notation}

\begin{rmk}\label{rmk:bundle as functor}
  An object in Definition~\ref{defn:prinbund1} is equivalent to a zig-zag~\cite{BCMNP} 
  \beq\label{eq:Gbundasfunctor0}
 X\xleftarrow{\sim} \tilde{Y} \rightarrow \pt\sq \cG
\eeq
for an essential equivalence $\tilde{Y}\xrightarrow{\sim} X$ of Lie groupoids and $\pt\sq \cG$ regarded as a discrete Lie 2-groupoid. Indeed, for $X$ a smooth manifold, any essential equivalence is equivalent to one of the form $\tilde{Y}\to X$, where $\tilde{Y}=\check{C}(Y)$ is the \v Cech groupoid associated to a surjective submersion $u \colon Y\to X$. A functor $\check{C}(Y) \to * \sq \cG$ is exactly the data $(u \colon Y \to X, \rho, \gamma)$, and a natural transformation of two such functors is exactly the data of a 1-morphism of the form $(Y, \id_Y, \id_Y, h, \eta)$. 
The zig-zags~\eqref{eq:Gbundasfunctor0} give \emph{flat} $\cG$-bundles as $\pt\sq\cG$ is a \emph{discrete} Lie 2-groupoid. 
\end{rmk}

\begin{rmk}\label{rmk:otherdefs}
As Notation~\ref{notation:calp} suggests, one can define a 2-group bundle as a stack $\cP\to X$ with $\cG$-action (e.g., see \cite{SP11}). As this paper only treats discrete 2-groups, the full apparatus of $\cG$-stacks is a bit overkill for our intended applications and we stick with the more hands-on definition of $\cG$-bundles above. 
\end{rmk}

\begin{rmk}\label{rmk:2gerbetriv2}
Definition~\ref{defn:prinbund1} can be rephrased in the language of higher differential geometry, following \cite[Theorem 1.4]{BCMNP}. The data $(u,\rho)$ provides a $G$-bundle $P_{u,\rho}\to X$. The 3-cocycle $\alpha\in Z^3(G;\U(1))$ determines a (flat) \emph{2-gerbe} over the stack $\pt\sq G$, and the data of $\gamma$ satisfying~\eqref{eq:highercompatcond2} is a trivialization of this 2-gerbe pulled back to $X\to \pt\sq G$ along the map classifying $P_{u, \rho}$. Trivializations of (flat) 2-gerbes form a bicategory. The 1-morphisms and 2-morphisms in Definition~\ref{defn:prinbund1} can be understood in terms of this bicategory of trivializations of 2-gerbes, but enhanced by pulling back the trivializations along isomorphisms of $G$-bundles $P_{u,\rho}\xrightarrow{\sim} P'_{u',\rho'}$ over $X$. 
\end{rmk}

\begin{ex}[The groupoid of ordinary $G$-bundles]\label{eq:ordinaryG}
When $A=\{e\}$ and $\cG=G$ is a finite discrete group, there is an equivalence
$$
\Bun_\cG(X)\simeq \Bun_G(X)
$$
with the 1-groupoid of ordinary $G$-bundles on $X$. Indeed, $\gamma$ is no additional data, and the data of~$\cP_{u,\rho,\gamma}$ is equivalent to $P_{u,\rho}$. 
\end{ex}

\begin{ex}[$A$-gerbes]\label{ex:Agerbes}
    In the case where $\cG=\pt\sq A$ for $A$ an abelian group, unpacking Definition~\ref{defn:prinbund1} identifies $\BibunGX$ with the standard cocycle description for the bicategory of (flat) $A$-gerbes on $X$,
    $$
    \Gerbe_A(X)\simeq \Bun_{\pt\sq A}(X). 
    $$
    The symmetric monoidal structure of $\Gerbe_A(X)$ given by tensoring $A$-gerbes corresponds to a symmetric monoidal structure on $\Bun_{* \sq A}(X)$ given by multiplying cocycles $\gamma, \gamma'$ (after pulling back to a refinement of covers). 
\end{ex}

\begin{lem}\label{lemma: pi}
There is a (forgetful) 2-functor
\beq\label{eq:functorpi}
\pi\colon \Bun_\cG(X)\to \Bun_G(X),
\eeq
which associates to each $\cG$-bundle $\cP_{u, \rho, \gamma}$ the underlying $G$-bundle $P_{u, \rho}$.
\end{lem}
\begin{proof}[Proof sketch.]
The value on objects being given above, it remains to specify the value on morphisms: the functor $\pi$ extracts the data $(Z,h)$ in~\eqref{eq: h condition} to construct an isomorphism of $G$-bundles $\varphi_{Z,h}\colon P_{u,\rho}\to P_{u',\rho'}$ via the \v{C}ech 1-cochain $h$ on $X$. We refer to \cite[Lemma 4.17]{BCMNP} for details.
\ep

\begin{rmk}\label{rmk: surfaces}
For a fixed $G$-bundle $P\to X$ and categorical central extension $\cG$ of $G$, it is possible that there are no $\cG$-bundles with underlying $G$-bundle $P$---in other words, the forgetful functor~\eqref{eq:functorpi} need not be essentially surjective. Indeed, there is an obstruction given by the class $[\rho^*\alpha]\in \rmH^3(X;\U(1))$ gotten by pulling back the degree 3 class $[\alpha]\in \rmH^3(\pt\sq G;\U(1))$ along the map $\rho\colon X\to \pt\sq G$ classifying $P$. Following Remark~\ref{rmk:2gerbetriv2}, this is precisely the obstruction to trivializing the 2-gerbe~$\rho^*\alpha$. However, when $X$ is a surface the cohomology group $\rmH^3(X;\U(1))$ vanishes for degree reasons. Hence, for surfaces the functor~\eqref{eq:functorpi} is essentially surjective. 
\end{rmk}

We can also give a higher holonomy description of $\Bun_\cG(X)$, compatible with the holonomy presentation of $\Bun_G(X)$ from Example \ref{ex: hol pres}.
Let $\cG= \cG(G, A, \alpha)$ be as in Notation~\ref{notation:GAalpha}. 

\begin{prop}\label{prop:weakrepn}
For $\Sigma$ a connected closed oriented surface of genus $g\ge 1$, the bicategory of flat $\cG$-bundles on~$\Sigma$ is equivalent to the bicategory of $\cG$-valued (weak) 2-homomorphisms $\pi_1\Sigma\to \cG$ and conjugations 
    \begin{align}\label{eq: holonomy equivalence}
    \BibunGSig \simeq \Bicat(\pt\sq\pi_1\Sigma, \pt\sq\cG).
    \end{align}
    Furthermore, under this equivalence, the forgetful functor $\pi \colon \BibunGSig\to\Bun_G(\Sigma)$, see ~\eqref{eq:functorpi}, becomes the functor that extracts an ordinary homomorphism from a 2-homomorphism
    $$
\pi_1\Sigma\to \cG\xrightarrow{\pi_0} G \implies \pi^\hol\colon \Hom(\pt\sq \pi_1\Sigma,\cG)\to \Hom(\pt\sq \pi_1\Sigma,\pt\sq G). 
    $$
\end{prop}
\bp
We compare $\Bun_\cG(\Sigma)$ from Definition~\ref{defn:prinbund1} with the bicategory $\Bicat(\pt\sq \pi_1\Sigma,\pt\sq \cG)$ as described explicitly in~\cite{GanterUsher}; below we recall the objects, referring to~\cite{GanterUsher} for details on 1- and 2-morphisms. 

As $\Sigma$ has genus $g\ge 1$, its universal cover $\tilde{\Sigma}$ is contractible and every $\cG$-bundle trivializes when pulled back along $\tilde{u} \colon \tilde{\Sigma}\to \Sigma$. Indeed, the $G$-bundle trivializes (so $\rho$ and hence $\rho^*\alpha$ are trivial); the choices of $\gamma$ comprise the set of 2-cocycles, and all such 2-cocycles are exact. Iterating the fiber products in~\eqref{eq:fiberedproductforuniv}, a $\cG$-bundle on $\Sigma$ defined relative to the universal cover provides the group theoretic data
\begin{align}\label{eq: definition of gammahol}
\begin{array}{cc}
    \rhohol \colon \pi_1 \Sigma \to G, & \qquad \rhohol(a) = \rho(y, y\cdot a)\\
    \gammahol \colon \pi_1 \Sigma \times \pi_1 \Sigma \to A, & \qquad \gammahol(a,b) = \gamma(y, y\cdot a, y \cdot ab).
    \end{array}
\end{align}
Then the conditions \eqref{eq: condition for rho} and \eqref{eq:highercompatcond2} on $\rho$ and $\gamma$ imply
\beq\label{eqn: hol condition for gamma}
\rhohol(gg')=\rhohol(g)\rhohol(g'), \quad \rhohol^*\alpha=d\gammahol, \quad \gammahol(a, 1) = \gammahol(1, b) = 1_G.
\eeq
The data and conditions~\eqref{eqn: hol condition for gamma} are precisely those of a strong monoidal functor $\pi_1\Sigma \to \cG$~\cite[Section 3.1.1]{GanterUsher}:  $\rhohol$ gives the value of the functor on objects and the value on morphisms is no data as $\pi_1\Sigma$ is a discrete 2-group. The map $\gammahol$ gives the compatibility with the monoidal structure, and the condition that $\rhohol^*\alpha=d\gammahol$ is the condition that $\gammahol$ is compatible with the associator for $\cG$.
This gives the equivalence $\Bicat(\pt\sq\pi_1\Sigma, \pt\sq\cG) \to \Bun_\cG(\Sigma)$ on the level of objects; we will denote the bifunctor in this direction by $\iota$ as we did in the case of ordinary holonomy for $G$-bundles. 

Similarly, a 1-isomorphism between flat $\cG$-bundles is equivalent to an element $\hhol\in G$ and a map $\etahol\colon \pi_1\Sigma\to A$ such that the composition $\tilde{\Sigma} \times \pi_1\Sigma \to \pi_1\Sigma \to A$ satisfy~\eqref{eq:eta condition}. Comparing with~\cite[Section 3.1.2]{GanterUsher}, this is equivalent to a natural isomorphism between functors from $\pi_1X\to \cG$. 

Finally, a 2-isomorphism between $\cG$-bundles is equivalent to an element $\omegahol\in A$, where a 2-isomorphism between 1-isomorphism exists if and only if $\eta=\eta'$ and $h=h'$. This recovers invertible modifications between natural isomorphisms of functors; see~\cite[Section~3.1.3]{GanterUsher}. 
This gives an equivalence of bicategories $\BibunGSig \simeq \Bicat(\pt\sq \pi_1\Sigma, \pt\sq \cG)$.
\ep

\begin{rmk}
The \v{C}ech complex relative to the universal cover $\tilde{\Sigma}\to \Sigma$ is the bar complex for $\pi_1\Sigma$ that computes group (co)homology. In this language, $\tilde{\gamma}$ in~\eqref{eqn: hol condition for gamma} is a coboundary of the group cohomology class $[\rho^*\alpha]\in H^3_{\rm grp}(\pi_1\Sigma;A)$, and $\tilde{\eta}$ changes this coboundary by an exact term. 
\end{rmk}

\begin{rmk}
We expect a version of the above proposition to hold for more general manifolds $X$; for this is it necessary to replace the fundamental group $\pi_1X$ with the \emph{fundamental 2-group} $\pi_{\le 2}X$. For $\Sigma$ as in the proposition, the fundamental 2-group is equivalent to the ordinary fundamental group. 
\end{rmk}

\section{Constructing a $\U(1)$-bundle over $\Bun_G(\Sigma)$ from the geometry of 2-group bundles}\label{sec: construct P}

The goal in this section is to use the functor $\pi \colon \Bun_\cG(X) \to \Bun_G(X)$ to construct a principal $\U(1)$-bundle $\sP_\cG$ over $\Bun_G(X)$, in the special case that $\cG= \cG(G, \U(1), \alpha)$ and $X$ is a connected closed oriented surface $\Sigma$ of genus $\ge 1$. In section~\ref{sec:proof of main thm} we will show that the line bundle associated to this principal $\U(1)$-bundle is isomorphic to the Freed--Quinn line bundle. 

Morally, the $\U(1)$-bundle $\sPG$ is the shadow of a higher categorical structure on $\Bun_\cG(\Sigma)$: indeed, for a general smooth manifold $X$ and 2-group of the form $\cG = \cG(G, A, \alpha)$, the bifunctor $\pi \colon \Bun_\cG(X) \to \Bun_G(X)$ is a 2-fibration over its essential image, which we denote by $B$. Furthermore, there is a natural action of the symmetric monoidal bicategory $\Gerbe_A(X)$ on $\Bun_\cG(X)$, given by twisting a principal bundle by a gerbe, making $\Bun_\cG(X)$ into a principal $\Gerbe_A(X)$-bundle over $B$. There are also actions of the group $\Diff(X)$ on the base $B$, the total space $\Bun_\cG(X)$, and the ``structure group'' $\Gerbe_A(X)$ (all given by pulling back along diffeomorphisms of $X$) making $\Bun_\cG(X) \to B$ into a $\Diff(X)$-equivariant principal $\Gerbe_A(X)$-bundle. 

It is beyond the scope of this article to define rigorously all of these notions, but we will take a truncation of the bicategory $\Bun_\cG(X)$ along the fibers of $\pi$, to produce a category $\sP_\cG$, reduce the ``structure group'' from the symmetric monoidal bicategory $\Gerbe_A(X)$ to the (ordinary) group of isomorphism classes of gerbes $\check{\rmH}^2(X;A)$, and prove that the resulting action of $\check{\rmH}^2(X;A)$ on the category $\cP_\cG$ does indeed produce a $\Diff(X)$-equivariant principal bundle.

Besides being essential for the comparison to the Freed--Quinn line bundle, the specialization from a general smooth manifold $X$ to an oriented surface $\Sigma$ yields three simplifications for us: (1) it ensures that the bifunctor $\pi$ is essentially surjective, so that its essential image $B$ is equal to $\Bun_G(\Sigma)$ (see Remark \ref{rmk: surfaces}; (2) it allows us to use the higher holonomy presentation of principal $\cG$-bundles (Proposition \ref{prop:weakrepn}), which simplifies some proofs; and (3) a choice of orientation of $\Sigma$ determines an isomorphism of groups $\check{\rmH}^2(\Sigma;A)\cong A$, so that we obtain a principal $A$-bundle over $\Bun_G(\Sigma)$. The specialization from a general abelian group $A$ to $\U(1)$ (with the discrete topology) is necessary only for comparing with Freed--Quinn, and does not otherwise affect any of the proofs. 

\subsection{The fibers of $\pi\colon \Bun_\cG(X)\to \Bun_G(X)$}

As a first step towards constructing a principal bundle over $\Bun_G(X)$ out of the functor $\pi$~\eqref{eq:functorpi}, we analyze the fibers of this functor. The results of this section hold in the general setting of a $2$-group $\cG(G, A, \alpha)$ and a smooth manifold $X$.

\begin{lem}\label{lem: gamma ratio is closed}
For a fixed $(u,\rho)\in \Bun_G(X)$, consider a pair of objects $(u,\rho,\gamma), (u,\rho,\gamma')\in \Bun_{\cG}(X)$. Then the ratio $\gamma/\gamma'$ of \v{C}ech 2-cochains determines a 2-cocycle, and hence has an underlying cohomology class
\beq\label{Eq:cohomclass}
[\gamma/\gamma']\in \check{\rmH}^2(X;A)
\eeq
in the \v{C}ech cohomology of $X$ with coefficients in $A$ (with its discrete topology). 
\end{lem}
\bp
Using the defining relation $\rho^*\alpha=d\gamma$ for $(u,\rho,\gamma)$ to be an object of $\Bun_{\cG}(X)$, we have
$$
d(\gamma/\gamma')=(\rho^*\alpha)/(\rho^*\alpha)=1, 
$$
as claimed. 
\ep
\begin{rmk}\label{remark about gerbes}
    In other words, Lemma \ref{lem: gamma ratio is closed} tells us that two objects $(u, \rho, \gamma)$ and $(u, \rho, \gamma')$ in the same fiber of $\pi$ differ by an $A$-gerbe. 
\end{rmk}

The following theorem provides a lifting of isomorphisms of $G$-bundles to isomorphisms of $\cG$-bundles. 

\begin{thm}\label{thm: lifting lemma}
Let $\varphi_{Y,h} \colon P_{u, \rho_1} \to P_{u, \rho_2}$ be an isomorphism of $G$-bundles defined with respect to the same surjective sumersion $u \colon Y \to X$, and let $\cP_{u, \rho_1, \gamma_1}$ be principal $\cG$-bundle living over $P_{u, \rho_1}$. 
\begin{enumerate}
    \item There exists a principal $\cG$-bundle $\cP_{u, \rho_2, {}^h\gamma_1}$ together with a 1-morphism $\varphi_{Y, h, \eta} \colon \cP_{u, \rho_1, \gamma_1} \to \cP_{u, \rho_2, {}^h\gamma_1}$ such that $\pi(\varphi_{Y, h, \eta}) = \varphi_{Y, h}$.

    \begin{center}
        \begin{tikzcd}
            \cP_{u,\rho_1,\gamma_1}\arrow[d, mapsto] \arrow[r, dashed, "\varphi_{Y,h,\eta}"]& \cP_{u,\rho_2, {}^h\gamma_1}\arrow[d, dashed, mapsto]\\
            P_{u,\rho_1} \arrow[r, "\varphi_{Y,h}"']& P_{u,\rho_2}
        \end{tikzcd}
    \end{center}
    \item For $\cP_{u, \rho_2, \gamma_2}$ another principal $\cG$ bundle living over $P_{u, \rho_2}$, there exists a 1-morphism $\varphi_{Y, h, \eta'} \colon \cP_{u, \rho_1, \gamma_1} \to \cP_{u, \rho_2, \gamma_2}$ with $\pi(\varphi_{Y, h, \eta})= \varphi_{Y, h}$ if and only if ${\gamma_2} = {}^h\gamma_1 d\eta''$ for some $\eta'' \colon Y\times_X Y \to A$; in particular, 
    \begin{align*}
        \left[{}^h\gamma_1/{\gamma_2}\right] = 1 \in \check{\rmH}^2(X;A).
    \end{align*}
\end{enumerate}
\end{thm}

\bp
The first statement amounts to the claim that there exist locally constant cochains ${}^h\gamma_1\colon Y^{[3]}\to A, \eta\colon Y^{[2]}\to A$, satisfying the conditions in Definition~\ref{defn:prinbund1}. Using the 3-cocycle condition for $\alpha$ and the relation $\rho_2 \cdot p_2^*h = p_1^*h \cdot \rho_1$, we find that 
\beq\label{eq:ratioofrho}
    \frac{\rho_1^*{\alpha}}{\rho_2^*{\alpha}}(y_1, y_2, y_3,y_4) = d\beta(y_1, y_2, y_3,y_4)
\eeq
for the 3-cochain $\beta$ defined as
\beq\label{eq:betaeq}
\beta(y_1,y_2,y_3)&=&\frac{\alpha(h(y_1), \rho_1(y_1, y_2), \rho_1(y_2, y_3)) \alpha(\rho_2(y_1, y_2), \rho_2(y_2, y_3), h(y_3))}{\alpha(\rho_2(y_1, y_2), h(y_2), \rho_1(y_2, y_3))}.
\eeq

Then define
\begin{align*}
{}^h\gamma_1(y_1,y_2,y_3) &= \gamma_1(y_1,y_2,y_3) \cdot \frac{\alpha(\rho_2(y_1, y_2), h(y_2), \rho_1(y_2, y_3))}{\alpha(h(y_1), \rho_1(y_1, y_2), \rho_1(y_2, y_3)) \alpha(\rho_2(y_1, y_2), \rho_2(y_2, y_3), h(y_3))}\\
&= \frac{\gamma_1}{\beta}(y_1,y_2,y_3).
\end{align*}

Using~\eqref{eq:ratioofrho} we find $d({}^h\gamma_1)=\rho_2^*\alpha$ and so $\cP_{u, \rho_2, {}^h\gamma_1}$ is an object of $\Bun_\cG(X)$. By definition, an isomorphism $\cP_{u, \rho_1,\gamma_1}\to \cP_{u, \rho_2, {}^h\gamma_1}$ covering the isomorphism of $G$-bundles is given by $\eta \colon Y^{[2]} \to A$ satisfying~\eqref{eq:eta condition}. From the definition of ${}^h\gamma_1$, this relation is satisfied for any $\eta$ with $d\eta=1$.

For the second statement, we consider the composition
\begin{align*}
    \cP_{u, \rho_2, \gamma_2}\xrightarrow{(\varphi_{Y, h, \eta'})^{-1}} \cP_{u, \rho_1, \gamma_1} \xrightarrow{\varphi_{Y, h, \eta}} \cP_{u, \rho_2, {}^h\gamma_1};
\end{align*}
we obtain $\varphi_{Y,1_G, \eta/\eta'} \colon \cP_{u, \rho_2, \gamma_2} \to \cP_{u, \rho_2, {}^h\gamma_1}$. Setting $\eta'' = \eta/\eta'$, condition \eqref{eq:eta condition} becomes $d\eta'' = \gamma_2/{}^h\gamma_1$, as claimed. Conversely, given $\eta''$ satisfying this condition, we take $\eta'= \eta/\eta''$ and set $\varphi_{Y, h, \eta'} = \varphi_{Y, 1_G, \eta''}^{-1} \circ \varphi_{Y, h, \eta}$. 
\ep

\begin{rmk}
    For $\Sigma$ a closed oriented surface of genus $g \ge 1$, we can restate Theorem \ref{thm: lifting lemma} in terms of the higher holonomy description of $\Bun_\cG(\Sigma)$. We will use the notation ${}^h\gammahol_1$ for the resulting lift, which is unique up to multiplication by an exact cocycle $d \tilde \eta$.
\end{rmk}

\begin{cor}\label{cor: lifting lemma with refinements}
    Let $\cP_{u, \rho, \gamma}, \cP_{u, \rho, \gamma'}$ be two principal $\cG$-bundles living over the principal $G$-bundle $P_{u, \rho}$, and defined with respect to the same surjective submersion $u \colon Y \to X$. Then there is an isomorphism $\varphi_{Z, 1_G, \eta} \colon \cP_{u, \rho, \gamma} \to \cP_{u, \rho, \gamma'}$ (possibly defined over a refinement of $Y$) covering $\id_{P_{u, \rho}}$ if and only if 
    \begin{align*}
        [\gamma/\gamma'] = 1 \in \check{\rmH}^2(X;A).
    \end{align*}
\end{cor}
\bp
This follows from part (2) of Theorem \ref{thm: lifting lemma} upon setting $h=1_G$ and allowing refinements of the cover $Y$. 
\ep

\begin{rmk}
    In the case that $X$ in a non-orientable surface, ${\rm H}^2(X; A)$ is trivial. Therefore it follows from Corollary \ref{cor: lifting lemma with refinements} that up to isomorphism there is a unique principal $\cG$-bundle living over any given principal $G$-bundle. 
\end{rmk}

\subsection{A groupoid $\sP_\cG$ over $\Bun_G(\Sigma)$ constructed from $\Bun_{\cG}(\Sigma)$}
We now define a truncation of the bicategory $\Bun_\cG(X)$ as follows:

\begin{defn}\label{defn:Pline}
Define the groupoid $\sPG$ as having objects equivalence classes $(u,\rho,[\gamma])$ for $\cP_{u,\rho,\gamma}\in \Bun_\cG(X)$ and the equivalence relation on the datum $\gamma$ defined by
\beq\label{eq:equivalencerelation}
\gamma\sim \gamma'\iff \text{ there is an isomorphism $\cP_{u, \rho, \gamma} \to \cP_{u, \rho, \gamma'}$ covering $\id_{P_{u, \rho}}$} . 
\eeq
Define $\sPG$ to have morphisms 
\beq\label{eq:morphismsinsP}
&&\Hom_{\sPG}((u_1,\rho_1,[\gamma_1]),(u_2,\rho_2,[\gamma_2])=\left\{\begin{array}{cc} \{\varphi_{Z,h}\colon P_{u_1,\rho_1}\to P_{u_2,\rho_2}\} & {\rm if} \ \exists \,\varphi_{Z,h,\eta}\colon \cP_{u_1,\rho_1,\gamma_1}\to \cP_{u_2,\rho_2,\gamma_2}\\ 
\emptyset & {\rm else}.\end{array}\right.
\eeq
Composition in $\sPG$ is inherited from composition in $\Bun_{\cG}(X)$. 
\end{defn}

\begin{rmk}\label{remark: equivalence relation in cohomology}
By Corollary \ref{cor: lifting lemma with refinements}, we have
\beq
\gamma\sim \gamma'\iff[\gamma/\gamma']=1\in \check{\rmH}^2(X;A).
\eeq
\end{rmk}

\begin{lem}
Definition~\ref{defn:Pline} does indeed produce a groupoid $\sPG$ that furthermore factors~\eqref{eq:functorpi} via
\begin{center}
    \begin{tikzcd}
          \Bun_\cG(X) \arrow[rr]\arrow[dr, "\pi"'] & & \sPG \arrow[dl, "\overline{\pi}"] \\
          & \Bun_G(X)
    \end{tikzcd}
\end{center}
for $\overline{\pi}(u,\rho,[\gamma])=P_{u,\rho}\in \Bun_G(X)$. 
\end{lem}
\bp
To see that the hom-sets~\eqref{eq:morphismsinsP} are well-defined, take different choices of representatives $(u_1,\rho_1,\gamma_1)$ and $(u_1, \rho_1, \gamma_1')$, respectively $(u_2, \rho_2, \gamma_2)$ and $(u_2, \rho_2, \gamma_2')$, i.e., there exist 1-morphisms
$$\cP_{u_1, \rho_1, \gamma_1} \simeq \cP_{u_1, \rho_1, \gamma_1'} \text{ and } \cP_{u_2, \rho_2, \gamma_2} \simeq \cP_{u_2, \rho_2, \gamma_2'}$$
covering the identity morphisms of the underlying $G$-bundles. Then by transitivity of isomorphisms, we see that a 1-morpshism $\cP_{u_1, \rho_1, \gamma_1}\xrightarrow{\sim} \cP_{u_2, \rho_2, \gamma_2}$ in $\Bun_\cG(X)$ exists if and only if there is a 1-isomorphism $\cP_{u_1, \rho_1, \gamma_1'}\xrightarrow{\sim} \cP_{u_2, \rho_2, \gamma_2'}$. Hence the hom sets are well-defined. Checking that composition is well-defined is routine. 

Finally, define the functor $\overline{\pi} \colon \sPG \to \Bun_G(X)$ as sending $(u, \rho, [\gamma]) \mapsto (u, \rho)$ on objects, and as the identity on morphisms. We observe that the functor $\pi \colon \Bun_\cG(X) \to \Bun_G(X)$ factors through $\sPG$.
\ep

\begin{rmk}
  For $X=\Sigma$ a connected closed oriented surface of genus $g \ge 1$, we also have a higher holonomy description of $\sP_\cG$, using Proposition \ref{prop:weakrepn}. That is, $\sP_\cG$ is equivalent to a groupoid $\sP_\cG^\hol$ with objects $({\rhohol}, [\gammahol])$ for $({\rhohol}, \gammahol) \colon * \sq \pi_1 \Sigma \to * \sq \cG$, where 
\begin{align*}
    \gammahol \sim \gammahol' \iff \text{ there is a natural transformation $ (1_G, \etahol) \colon ({\rhohol}, \gammahol) \to ({\rhohol}, \gammahol')$.} 
\end{align*}
The morphisms in $\sP_\cG^\hol$ are defined analogously to the definition of $\sP_\cG$. The equivalence $\iota \colon \Bicat(\pt\sq\pi_1\Sigma, \pt\sq\cG) \to \Bun_\cG(\Sigma)$ of Proposition \ref{prop:weakrepn} induces an equivalence $\overline{\iota} \colon \sP_\cG ^\hol \to \sP_\cG$, under which the functor $\overline{\pi}$ corresponds to a functor $\overline{\pi}^\hol$ which sends $(\rhohol, [\gammahol])$ to $P_\rhohol \in \Bun_G(\Sigma).$ 
\end{rmk}

\begin{ex}\label{ex:U(1)gerbe}
This is a continuation of Example~\ref{ex:Agerbes}, where we observed that for $\cG=\pt\sq \U(1)$, we have an equivalence $\Bun_{\pt\sq \U(1)}(X) \simeq \Gerbe_{\U(1)}(X)$. 
Multiplication of cocycles $\gamma, \gamma'$ induces a monoidal structure on $\sP_{* \sq \U(1)}$, corresponding to the tensor product of flat $\U(1)$-gerbes. The following boils down to the standard fact that isomorphism classes of (flat) gerbes on a manifold $X$ are in bijection $\check\rmH^2(X;\U(1))$.
\end{ex}

\begin{prop}\label{prop: PU(1) is H2 or U(2)}
\begin{enumerate}
\item Viewing $\check{\rmH}^2(X;\U(1))$ as a discrete category with the monoidal structure given by the usual group operation, there is a canonical strict equivalence of monoidal categories $\sP_{* \sq \U(1)} \simeq \check{\rmH}^2(X;\U(1))$.
\item If $X=\Sigma$ is an orientable surface, an orientation of $\Sigma$ determines a strict equivalence of monoidal categories
$\sP_{* \sq \U(1)} \simeq \U(1)$, where $\U(1)$ is viewed as a discrete monoidal category under the usual multiplication. 
\end{enumerate}
\end{prop}

\bp
We define a functor $\sP_{*\sq \U(1)} \to \check{\rmH}^2(X;\U(1))$ by sending $(u, [\gamma])$ to the class of the cocycle $\gamma$ in $ \check{\rmH}^2(X;\U(1))$. By Remark \ref{remark: equivalence relation in cohomology}, this is well-defined. Furthermore, we note that $\Hom_{\sP_{*\sq \U(1)}}((u_1, [\gamma_1]), (u_2, [\gamma_2])$ is empty unless $[\gamma_1] =[\gamma_2]$, in which case it contains a single morphism corresponding to the identity automorphism of the principal $\{*\}$-bundle $X \times \{*\}\to X$. Our functor will then send this morphism to $\id_{[\gamma_1]}$. It then follows directly that this functor is essentially surjective and fully-faithful. It is also clear that it is compatible with the monoidal structures; strictness is immediate because the target is a discrete category.

Finally, in the case that $X=\Sigma$ is 2-dimensional and oriented we use that $\check{\rmH}^2(\Sigma;\U(1))\simeq \U(1)$, where the isomorphism of groups is specified by the choice of orientation. 
\ep

\begin{rmk}
In fact, in the case of $X = \Sigma$ a closed oriented surface, the monoidal category $\sP_{* \sq \U(1)}^\hol$ is \emph{equal} to the quotient of the group $2$-cocycles by exact $2$-cocyles, i.e. to the group cohomology $\rmH^2(\pi_1\Sigma; \U(1))$, and hence (because $\Sigma$ is a $K(\pi, 1)$) to the singular cohomology $\rmH^2(\Sigma; \U(1))$. 
\end{rmk}

\subsection{The $\U(1)$-action on $\sP_\cG$}
From now on, let $\cG= \cG(G, \U(1), \alpha)$, and let $\Sigma$ be a connected closed oriented surface of genus $g \ge 1$. In this subsection, we will construct an action of $\U(1)$ (or more specifically, $\sP_{*\sq\U(1)}$) on $\sP_\cG$. Morally, this is induced from a higher categorical action of $\Gerbe_{\U(1)}(\Sigma)$ on $\Bun_\cG(\Sigma)$, defined by twisting a principal bundle by a gerbe (see Proposition 4.24 of \cite{BCMNP}). 

For covers $u\colon Y\to \Sigma$ and $u'\colon Y'\to \Sigma$, let $u\cdot u'$ denote the mutual refinement $Y\times_\Sigma Y'\to \Sigma$. More generally, we recall from Example \ref{ex:Cech} the notation for a refinement of covers $u_i\colon Y_i\to \Sigma$,
$$
Y_1 \xleftarrow{v_1} Z\xrightarrow{v_2} Y_2,\quad  u_1\circ v_1=u_2\circ v_2\colon Z\to \Sigma.
$$

\begin{defn}\label{defn:action of U on sPG}
Define the functor $\gact \colon \sPG \times \sP_{*\sq\U(1)} \to \sPG$ whose value on objects and morphisms is 
\beq\label{eq:actfunctor}
    \begin{array}{ccc} \gact((u,\rho,[\gamma]),(u',[\gamma']))&:=& (u\cdot u',\rho,[\gamma\cdot \gamma'])\\
    \gact((Z,v_1, v_2, h), (Z', v_1', v_2')) &:=& (Z\times_\Sigma Z',v_1\cdot v_1', v_2\cdot v_2', h).
    \end{array}
\eeq
where $\gamma, \gamma', \rho$ and $h$ have all been pulled back to the refinement $u\cdot u'$ of the covers $u$ and $u'$. 
\end{defn}

\begin{lem}
    The assignments~\eqref{eq:actfunctor} determine a functor. 
\end{lem}

\bp
It is easy to see that this is well-defined on equivalence classes of objects. 

Given morphisms $(Z,v_1, v_2, h)\colon (u_1, \rho_1, [\gamma_1]) \to (u_2, \rho_2, [\gamma_2])$ in $\sPG$ and $(Z', v_1', v_2')\colon (u_1',[\gamma_1'])\to (u_2',[\gamma_2'])$ in $\sP_{\pt\sq\U(1)}$, by definition there exist 1-morphisms
\begin{align}\label{eq: choice of representatives}
\varphi_{Z,h,\eta}\colon \cP_{u_1,\rho_1,\gamma_1}\to \cP_{u_2,\rho_2,\gamma_2}\in \Bun_\cG(\Sigma),\qquad \varphi_{Z',\eta'}\colon \mathcal{A}_{u_1',\gamma_1'}\to \mathcal{A}_{u_2',\gamma_2'}\in \Gerbe_{\U(1)}(\Sigma) 
\end{align}
From these lifts, we form the 1-morphism in $\Bun_\cG(\Sigma)$~\cite[Proposition 4.24]{BCMNP}
$$
\varphi_{Z\times_\Sigma Z',h,\eta\cdot \eta'}\colon \cP_{u_1\cdot u_1',\rho_1,\gamma_1\cdot\gamma_1'}\to \cP_{u_2\cdot u_2',\rho_2,\gamma_2\cdot \gamma_2'}
$$
where $u_i\cdot u_i'\colon Y_i\times_\Sigma Y_i'\to \Sigma$ are the refinements of covers. The existence of $\varphi_{Z\times_\Sigma Z',h,\eta\cdot \eta'}$ implies that $(Z\times_\Sigma Z',v_1 \cdot v_1', v_2 \cdot v_2', h)$ is indeed a morphism $(u_1\cdot u_1',\rho_1,[\gamma_1\cdot\gamma_1'])\to  (u_2\cdot u_2',\rho_2,[\gamma_2\cdot\gamma_2'])$ in $\sPG$. It is easy to see that \eqref{eq:actfunctor} preserves identity morphisms and composition. 
\ep

 \begin{ex}\label{ex: action for same cover}
     For $(u,\rho,[\gamma])\in \sPG$ and $(u',[\gamma'])\in\sP_{*\sq \U(1)}$, if the surjective submersions agree, i.e. $u=u'$, then $\gact((u, \rho, [\gamma]), (u', [\gamma'])) \cong (u, \rho, [\gamma\cdot \gamma'])$ in $\sPG$. 

     Indeed, by definition $\gact((u, \rho, [\gamma]), (u, [\gamma'])) = (Y \times_\Sigma Y \to \Sigma, p_1^* \rho, [p_1^* \gamma \cdot p_2^*\gamma'])$;  however, $Y \to \Sigma$ factors through the diagonal embedding as 
     \begin{align*}
         Y \xrightarrow{\Delta} Y \times_\Sigma Y \xrightarrow{u\cdot u} Y,
     \end{align*}
     so that 
     \begin{align*}
         \gact((u, \rho, [\gamma]), (u, [\gamma']))  & = (Y \times_\Sigma Y \to \Sigma, p_1^* \rho, [p_1^* \gamma \cdot p_2^*\gamma']) \\
         & \cong \Delta^*(Y \times_\Sigma Y \to \Sigma, p_1^* \rho, [p_1^* \gamma \cdot p_2^*\gamma']) \\
         & = (Y \to \Sigma, \rho, [\gamma \cdot \gamma'])
     \end{align*}
     as claimed. 
\end{ex}

Using the natural equivalences $\sP^\hol_\cG \simeq \sP_\cG$ and $\sP^\hol_{*\sq \U(1)} \simeq \sP_{*\sq \U(1)}$, the induced action functor in the holonomy presentations behaves as follows.
\begin{defn}
    Define the functor $\gact^\hol\colon\sP^\hol_\cG\times \sP^\hol_{*\sq\U(1)}\to \sP^\hol_\cG$ whose value on objects and morphisms is

   \beq\label{eq:actholfunctor}
    \begin{array}{ccc} \gact^\hol((\rhohol, [\gammahol]),([\gammahol']))&:=& (\rhohol,[\gammahol\cdot \gammahol'])\\
    \gact^\hol(\hhol, 1) &:=& \hhol.
    \end{array}
\eeq    
\end{defn}

\begin{lemma}
    Under the equivalence $\sP^\hol_{*\sq \U(1)} \to \sP_{*\sq \U(1)}$, the equivalence $\overline{\iota} \colon \sP^\hol_\cG \to \sP_\cG$ is equivariant. 
\end{lemma}
\bp
This follows from the definition of the actions, together with Example \ref{ex: action for same cover}.
\ep

From this, the following is immediate.

\begin{lem}\label{prop: H2 and U1 act}
    The functor $\gact^\hol \colon \sP^\hol_\cG \times \sP^\hol_{* \sq \U(1)} \to \sP^\hol_\cG$ determines a strict action of $\sP^\hol_{*\sq\U(1)}$ on $\sP^\hol_\cG$.

    Hence, the functor $\gact\colon \sPG\times \sP_{*\sq\U(1)}\to \sPG$ gives a (weak) action of the monoidal category $\sP_{*\sq\U(1)}$ on the category $\sPG$, and by Proposition~\ref{prop: PU(1) is H2 or U(2)} the group $\U(1)$ acts on $\sP_\cG$.
\end{lem}

\begin{prop}\label{prop:U(1)bundle}
    The functor $\gact\colon \sPG\times\sP_{\pt\sq\U(1)} \to\sPG$ witnesses $\sPG$ as a principal $U(1)$-bundle over $\Bun_G(\Sigma)$.
\end{prop}

\bp
See Definition~\ref{defn:U1bundle} for the definition of a $\U(1)$-bundle over the groupoid $\Bun_G(\Sigma)$. We take the epimorphism $\cc\colon \check{C}^1(\Sigma,G)\to \Bun_G(\Sigma)$ as in Example~\ref{ex:Cech}, and it remains to show that the pullback $\cc^*\sP_\cG=\cY$ is equivariantly equivalent to $\check{C}^{1}(\Sigma,G)\times \U(1)$. In fact we will show that it is equivariantly equivalent to $\check{C}^{1}(\Sigma, G) \times \sP_{* \sq U(1)}$ and then compose with the monoidal equivalence $\sP_{* \sq U(1)} \simeq \U(1)$. Consider the 2-pullback square

 \begin{center}
     \begin{tikzcd}
     \cY \arrow[r]\arrow[d] & \sP_\cG\arrow[d]\\
     \check{C}^1(\Sigma,G) \arrow[r]& \Bun_G(\Sigma).
    \end{tikzcd}
 \end{center}

 We begin by unpacking the pullback $\cY$. The objects of $\cY$ consist of triples $(u,\rho)\in \check{C}^1(\Sigma,G), (u',\rho', [\gamma'])\in\sP_\cG$ and $\iso\colon P_{u',\rho'}\to P_{u,\rho}$ in $\Bun_G(\Sigma)$. As $\check{C}^1(\Sigma,G)$ is a set, a morphism in $\cY$ consists of a morphism $(Z,v_1,v_2,h)\colon (u'_1,\rho'_1, [\gamma'_1])\to (u'_2,\rho'_2,[\gamma'_2])$ in $\sP_\cG$ that is compatible with the morphisms $\iso_1,\iso_2$; i.e., it fits in a commutative diagram
\begin{center}
    \begin{tikzcd}
        P_{u_1', \rho'_1,}\arrow[rr, "\iso_{Z,h}"'] \arrow[dr, "\iso_1"'] & &P_{u'_2,\rho'_2}\arrow[dl,"\iso_2"]\\
        & P_{u, \rho}.
    \end{tikzcd}
\end{center}
We remark that because of the requirement imposed by this diagram, the morphisms between each pair of objects of $\cY$ are unique if they exist; thus $\cY$ is equivalent to a set. The action of $\sP_{* \sq \U(1)}$ on $\cY$ is given as follows: 
\begin{align*}
    ((u, \rho), (u', \rho', [\gamma']), \varphi) \cdot (u'', [\gamma'']) = ((u, \rho), (u'\cdot u'', \rho', [\gamma'\cdot \gamma'']), \varphi \circ \varphi_{Y' \times_X Y'', \proj_{Y'}, \id_{Y' \times_X Y''}, 1_g}^{-1} ),
\end{align*}
where $\varphi_{Y' \times_X Y'', \proj_{Y'}, \id_{Y' \times_X Y''}, 1_g} \colon P_{u', \rho'} \to P_{u'\cdot u'', \rho'}$ is the canonical isomorphism of Example \ref{rmk: pullback gives morphism}.

To define a functor from $\check{C}^1(\Sigma, G) \times \sP_{* \sq U(1)}$ to $\cY$, we choose a section of $\overline{\pi}$ over $\check{C}^1(\Sigma, G)$, which is equivalent to a functor $\sigma \colon \check{C}^1(\Sigma, G) \to \sP_\cG$ that for each $(u, \rho) \in \check{C}^1(\Sigma, G)$ fixes $(u, \rho, [\gamma_\rho]) \in \sPG$. Then we consider the diagram

\begin{center}
    \begin{tikzcd}
    \check{C}^1(\Sigma,G)\times\sP_{* \sq U(1)}\arrow[ddr, bend right=20, "p_1"']\arrow[dr, dashed] \arrow[r, "\sigma"] & \sPG \times \sP_{* \sq \U(1) \arrow[dr, "\gact"]}&\\
        &\cY \arrow[r]\arrow[d] & \sP_\cG\arrow[d, "\overline{\pi}"]\\
        &\check{C}^1(\Sigma,G) \arrow[r]& \Bun_G(\Sigma).
    \end{tikzcd}
\end{center}

It is easy to see that the outer square commutes weakly, with 2-commuting data given by the morphisms induced by refinements of covers from Example \ref{rmk: pullback gives morphism}; hence this diagram induces the dashed arrow. Because the composition $\check{C}^1(\Sigma, G) \times \sP_{* \sq U(1)} \to \sPG$ is $\sP_{* \sq U(1)}$-equivariant, the dashed arrrow is too. It remains to show that it is an equivalence of categories. 

We can see that it is essentially surjective as follows. Given an object $((u, \rho), (u', \rho', [\gamma'], \varphi)) \in \cY$, choose representatives $\gamma'$ of the class $[\gamma']$ and $(Z, v_1, v_2, h)$ of the morphism $\varphi$, and pull the data $\rho, \rho', \gamma'$ back to the cover $Z$. Theorem \ref{thm: lifting lemma} provides lifts $(u, \rho, {}^h\gamma')$ of $(u, \rho)$ and $\varphi_{Z, h, \eta} \colon \cP_{u', \rho', \gamma'} \to \cP_{u, \rho, {}^h\gamma'}$ of $\varphi_{Z, h}$. In particular, $(Z, [{}^h\gamma'/\gamma_\rho]) \in \cP_{* \sq \U(1)}$, so that $((u, \rho), (Z \to \Sigma,  [{}^h\gamma'/\gamma_\rho]))$ is an object of $\check{C}^1(\Sigma, G) \times \sP_{* \sq \U(1)}$; it is straightforward to see that its image under the dashed arrow is isomorphic to $((u, \rho), (u', \rho', [\gamma'], \varphi)) \in \cY$. 

It is also easy to check that the dashed arrow is fully faithful: since the hom-sets of both source and target are either empty or singleton sets, it suffices to observe that the dashed arrow sends a non-isomorphic pair of objects in the source to a non-isomorphic pair of objects in the target. 
\ep

\begin{rmk}
    Morally, the ideas of this proof work on the level of the $\Gerbe_{\U(1)}(\Sigma)$-action on $\Bun_\cG(\Sigma)$: beginning with a section $\sigma \colon \check{C}^1(\Sigma, G) \to \Bun_\cG(\Sigma)$ of $\pi$, one obtains a trivialization of the pullback of $\Bun_\cG(\Sigma)$ to $\check{C}^1(\Sigma, G)$. 
\end{rmk}

\subsection{The mapping class group action on $\sPG$} 

From \cite[Corollary~4.11]{BCMNP}, the bicategory of $\cG$-bundles forms a 2-stack, and in particular $\cG$-bundles and automorphisms of $\cG$-bundles pull back along maps between smooth manifolds. This defines an action of $\Diff(\Sigma)$ on $\Bun_\cG(\Sigma)$, which descends to an action on $\sP_\cG$ as follows. 

\begin{defn}
Viewing the set $\Diff(\Sigma)$ as a discrete category, define a functor
\beq\label{eq:Diffaction}
\sP_\cG\times \Diff(\Sigma) \to \sP_\cG,\qquad ((u,\rho,[\gamma]), f)\mapsto (f^{-1}\circ u,\rho,[\gamma])
\eeq
that on representatives of equivalence classes pulls back $\cG$-bundles and $\cG$-bundle isomorphisms along a diffeomorphism $f\colon \Sigma\to \Sigma$. Here we identify $\Sigma \times_{f,\Sigma, u} Y \to \Sigma$ with $Y \xrightarrow{f^{-1} \circ u} \Sigma$ via the projection map $\Sigma \times_{f, \Sigma, u} Y \to Y$.

On morphisms, note that at the level of $\cG$-bundles, we have
\beq
&&\Big((Z, v_1, v_2, (h, \eta)) \colon (u_1 \colon Y_1 \to \Sigma, \rho_1, \gamma_1) \to (u_2 \colon Y_2 \to \Sigma, \rho_2, \gamma_2)\Big)\cdot f\nonumber\\
&&=\Big((Z, v_1, v_2, (h, \eta)) \colon (f^{-1} \circ u_1 \colon Y_1 \to \Sigma, \rho_1, \gamma_1) \to (f^{-1}\circ u_2 \colon Y_2 \to \Sigma, \rho_2, \gamma_2)\Big).\nonumber
\eeq
This uses the equality $f^{-1} \circ u_1 \circ v_1 = f^{-1} \circ u_2 \circ v_2$ which follows from $Z$ being compatible with the covers $u_1$ and $u_2$, i.e., $u_1 \circ v_1 = u_2 \circ v_2$. Taking equivalence classes in $\sP_\cG$ removes the data $\eta$ and passes to a map between the equivalence classes associated to $[\gamma_1]$ and $[\gamma_2]$: 
\begin{align*}
    \Big((Z, v_1, v_2, h) \colon (u_1, \rho_1, [\gamma_1]) \to (u_2, \rho_2, [\gamma_2] \Big) \cdot f = \Big((Z, v_1, v_2, h) \colon (f^{-1} \circ u_1, \rho_1, [\gamma_1]) \to (f^{-1} \circ u_2, \rho_2, [\gamma_2]) \Big).
\end{align*}
\end{defn}
The following is then immediate. 

\begin{lem}
The functor \eqref{eq:Diffaction} is well-defined: the pullback of $\cG$-bundles and $\cG$-bundle isomorphisms is well-defined on equivalence classes in $\sP_\cG$. Furthermore, this functor gives a strict action of $\Diff(\Sigma)$ on $\sP_\cG$, and the functor $\overline{\pi} \colon \sPG \to \Bun_G(\Sigma)$ is strictly $\Diff(\Sigma)$-equivariant. 
\end{lem}

\begin{defn}\label{eq: diff action on hol P}
    Viewing the set $\Diff(\Sigma)$ as a discrete category, define a functor 
    \begin{align*}
        \sPG^\hol \times \Diff(\Sigma) & \to \sPG^\hol\\
        \text{ on objects: } ((\rhohol, [\gammahol]), f) & \mapsto (\rhohol \circ f_*, [\gammahol \circ (f_*\times f_*)]) \\
        \text{ on morphisms: } (\hhol, f) & \mapsto (\hhol \circ f_*),
    \end{align*}
    for $f_* \colon \pi_1 \Sigma \to \pi_1 \Sigma$ the induced homomorphism. 
\end{defn}

\begin{lem}
    The functor \eqref{eq: diff action on hol P} gives a well-defined functor.
\end{lem}
\begin{proof}
    Suppose that $(\rhohol, [\gammahol_1]) = (\rhohol, [\gammahol_2])$. Then $\gammahol_1 / \gammahol_2 = d \etahol$, which implies that $(\gammahol_1 \circ (f_* \times f_*) / (\gammahol _2\circ (f_* \times f_*)) = d (\etahol \circ f_*)$. Thus $(\rhohol \circ f_*, [\gammahol_1 \circ (f_*\times f_*)]) = (\rhohol \circ f_*, [\gammahol_2 \circ (f_*\times f_*)])$, and the functor is well-defined on objects. It is clear that it is well-defined on morphisms, and that it respects identity morphisms and composition. 
\end{proof}

\begin{lem}\label{lemma: diff action on P-hol}
 The functor \eqref{eq: diff action on hol P} gives a strict action of $\Diff(\Sigma)$ on $\sPG^\hol$ such that the equivalence $\overline{\iota} \colon \sPG^\hol \to \sPG$ is (weakly) equivariant. Furthermore, the functor $\overline{\pi}^\hol \colon \sPG^\hol \to \Bun_G(\Sigma)$ is strictly equivariant. 
\end{lem}

\begin{proof}
    It is easy to see from the definitions that the functor gives a strict action of $\Diff(\Sigma)$ on $\sPG^\hol$ and that $\overline{\pi}^\hol$ is strictly equivariant, so it remains to check that the equivalence $\sPG^\hol \to \sPG$ is equivariant. The idea of the proof is similar to that of Proposition \ref{prop: holonomy equivariance}. Given an object $(\rhohol,[\gammahol]) \in \sPG^\hol$, represented by the weak homomorphism $(\rhohol, \gammahol)$, and a diffeomorphism $f \colon \Sigma \to \Sigma$, we need to provide isomorphisms $\overline\iota(\rhohol,[\gammahol]) \cdot f \to \overline\iota((\rhohol, [\gammahol]) \cdot f)$. Denoting $\overline\iota(\rhohol, [\gammahol])$ by $(\tilde{u}, \rho, [\gamma])$, we have $\overline\iota(\rhohol,[\gammahol]) \cdot f = (f^{-1} \circ \tilde u , \rho, [\gamma])$. It is naturally isomorphic to its pullback along a lift $\tilde{f} \colon \tilde X \to \tilde X$ of $f$, which is $(f^{-1} \circ \tilde u \circ \tilde f, \tilde f^* \rho, [\tilde f^* \gamma])$. A quick calculation analogous to that of the proof of Proposition \ref{prop: holonomy equivariance} shows that $(f^{-1} \circ \tilde u \circ \tilde f, \tilde f^* \rho, [\tilde f^* \gamma])= \overline\iota (\rhohol \circ f_*, [\gammahol \circ (f_* \times f_*)])$. It is straightforward to show that these equivalences are compatible with composition of diffeomorphisms. 
\end{proof}

\begin{rmk}
    In particular, the value of $f_*$ depends only on the class of $f$ in the mapping class group, and so the action of $\Diff(\Sigma)$ on $\sPG^\hol$ factors through the mapping class group of $\Sigma$. 
\end{rmk}

\begin{ex}\label{ex: diff acts trivially on U(1)}
    The action of $\Diff(\Sigma)$ on $\sP_{*\sq \U(1)}^\hol = \rmH^2(\Sigma; \U(1))$ is trivial. Indeed, orientation-preserving diffeomorphisms act trivially on the top cohomology group, so that $[\gammahol \circ (f_* \times f_*)]= [\gammahol]$.
\end{ex}

\begin{lem}\label{lem: actions commute}
    The actions of $\sP_{* \sq \U(1)}^\hol$ and $\Diff(\Sigma)$ on $\sPG^\hol$ commute, and hence the actions of $\sP_{* \sq \U(1)}$ and $\Diff(\Sigma)$ on $\sPG$ commute weakly.
\end{lem}
\begin{proof}
    Given objects $(\rhohol, [\gammahol]) \in \sPG^\hol, [\gammahol'] \in \sP_{*\sq \U(1)}$, and $f \in \Diff(\Sigma)$, we calculate
    \begin{align*}
        ((\rhohol,[\gammahol])\cdot [\gammahol']) \cdot f & = (\rhohol \circ f_*, [(\gammahol \gammahol') \circ (f_* \times f_*)]) = (\rhohol \circ f_*, [\gammahol\circ (f_* \times f_*)]) \cdot[ \gammahol' \circ (f_* \times f_*)]; \\
        ((\rhohol,[\gammahol])\cdot f) \cdot [\gammahol'] & = 
        (\rhohol \circ f_*, [(\gammahol \circ (f_* \times f_*)) \gammahol']).
    \end{align*}
    By Example \ref{ex: diff acts trivially on U(1)}, these are equal. Since $\sP_{* \sq \U(1)}^\hol$ is equivalent to a set, there is nothing to check at the level of morphisms. 
\end{proof}

\begin{thm}
    The principal $\U(1)$-bundles $\sPG$ and $\sPG^\hol$ are (equivalent) $\Diff(\Sigma)$-equivariant bundles over $\Bun_G(\Sigma)$. 
\end{thm}

\begin{proof}
    By the first part of Lemma \ref{lemma: diff action on P-hol} and by Lemma \ref{lem: actions commute}, we conclude that $\overline{\pi}^\hol \colon \sPG^\hol \to \Bun_G(\Sigma)$ is a $\Diff(\Sigma)$-equivariant principal $\sP_{* \sq \U(1)}$-bundle. Then using that $\overline{\iota}$ is $\Diff(\Sigma)$-equivariant and intertwines the actions of $\sP_{* \sq \U(1)}^\hol$ and $\sP_{* \sq \U(1)}$, we obtain that $\overline{\pi} \colon \sPG \to \Bun_G(\Sigma)$ is a $\Diff(\Sigma)$-equivariant principal $\sP_{* \sq \U(1)}$-bundle. 
    Finally, a choice of orientation of $\Sigma$ allows us to replace the structure group $\sP_{* \sq \U(1)}$ with $\U(1)$, as in the second part of Proposition \ref{prop: PU(1) is H2 or U(2)}.
\end{proof}

\subsection{Cocycles for $\sPG$}
To extract explicit formulas that characterize the bundle $\sPG\to \Bun_G(\Sigma)$ (and ultimately, compare with Freed and Quinn), we use the general setup~\S\ref{subsection: equivariant cocycles} to compute an equivariant cocycle for $\sPG$ relative to an equivariant cover of the groupoid $ \Bun_G(\Sigma)$. We do this for our three favorite presentations of $\Bun_G(\Sigma)$, namely, the \v Cech, triangulation, and holonomy presentations. 

We begin by analyzing such cocycles for the cover $\cc\colon \check{C}^1(\Sigma,G)\to \Bun_G(\Sigma)$ defined in Example~\ref{ex:Cech}. A trivialization of $\sPG$ over this cover is the data of a section of the pulled-back $\U(1)$-bundle $\cc^*\sP_\cG$; as in the proof of Proposition \ref{prop:U(1)bundle}, for each $(u, \rho) \in \check{C}^1(\Sigma,G)$, fix (once and for all) a choice of 2-cochain $\gamma_{u, \rho}$ with $d\gamma_{u, \rho} = \rho^* \alpha$, and define 
$$\label{eq:trivsigma} \sigma \colon \check{C}^1(\Sigma, G) \to \cc^*\sPG, \qquad \sigma(u, \rho) := ((u, \rho),(u, \rho, [\gamma_{u,\rho}]), \id_{P_{u, \rho}})\in \cc^*\sPG. 
$$
The data $\gamma_{u,\rho}$ exist because $\Sigma$ is a 2-manifold; see remark \ref{rmk: surfaces}. 

\begin{lem} \label{lem:cocycleCechsPG}
The cocycle for the $\U(1)$-bundle $\sPG\to \Bun_G(\Sigma)$ relative to the section~\eqref{eq:trivsigma} is
$$
R\colon \check{C}^1(\Sigma,G)\times_{\Bun_G(\Sigma)}\check{C}^1(\Sigma,G) \to \U(1),\quad R(\varphi_{Z,h})=[{}^h v_1^*\gamma_{u_1,\rho_1}/v_2^*\gamma_{u_2,\rho_2}]\in \rmH^2(\Sigma;\U(1))\simeq \U(1)
$$
for $\varphi_{Z,h} = \varphi_{Z, v_1, v_2, h}\colon P_{u_1,\rho_1}\to P_{u_2,\rho_2}$ an isomorphism of $G$-bundles defined over a refinement $v_i \colon Z \to Y_i$ of the covers $u_i$.
Here ${}^h(v_1^*\gamma_{u_1, \rho_1})$ is defined using Theorem \ref{thm: lifting lemma}.
\end{lem}

\bp 
As in \ref{subsection: cocycles}, $R(\varphi_{Z,h}) \in \check{\rmH}^2(\Sigma; U(1))$ is the unique value such that there exists an isomorphism in the (weak) fiber $\sP_{\cG,(u_2, \rho_2)}$ between $\sP_\cG(\varphi_{Z,h})(\sigma(u_1, \rho_1))$ and $\sigma(u_2, \rho_2) \cdot R(\varphi_{Z,h})$.

We have
\begin{align*}
  \sPG(\varphi_{Z,h})(\sigma(u_1, \rho_1)) & =((u_2, \rho_2), (u_1, \rho_1,[\gamma_{u_1, \rho_1}]), \varphi_{Z, v_1, v_2, h}) \\
  & \simeq ((u_2, \rho_2), (u_1 \circ v_1, v_1^* \rho_1, [v_1^*\gamma_{u_1, \rho_1}]), \varphi_{Z, \id_Z, v_2, h}) \\
  & \simeq ((u_2, \rho_2), (u_2 \circ v_2, v_2^* \rho_2, [{}^hv_1^*\gamma_{u_1, \rho_1}]), \varphi_{Z, \id_Z, v_2, 1_G}).
\end{align*}
here the first isomorphism comes from the canonical isomorphism between an object and its pullback along a refinement of the covers, and the second comes from Theorem \ref{thm: lifting lemma} applied to the isomorphism $\varphi_{Z, \id_Z, \id_Z, h}$. 

On the other hand, 
\begin{align*}
    \sigma(u_2, \rho_2) \cdot [{}^hv_1^*\gamma_{u_1, \rho_1} / v_2^* \gamma_{u_2, \rho_2}] & \simeq ((u_2, \rho_2), (u_2 \circ v_2, v_2^*\rho_2, [v_2^* \gamma_{u_2, \rho_2}]), \varphi_{Z, \id_Z, v_2, 1_G}) \cdot [{}^hv_1^*\gamma_{u_1, \rho_1} / v_2^* \gamma_{u_2, \rho_2}], 
\end{align*}
and by Example \ref{ex: action for same cover}, this is isomorphic to 
\begin{align*}
    ((u_2, \rho_2), (u_2 \circ v_2, v_2^*\rho_2, [{}^h v_1^* \gamma_{u_1, \rho_1}]), \varphi_{Z, \id_Z,v_2, 1_G}),
\end{align*}
as required. 

\ep

Recall from Proposition \ref{prop: cech equivariance} that the cover $\cc\colon \check{C}^1(\Sigma,G)\to \Bun_G(\Sigma)$ is $\Diff(\Sigma)$-equivariant. We are therefore in the set-up of \ref{subsection: equivariant cocycles}, and can encode the $\Diff(\Sigma)$-equivariant structure of the $\U(1)$-bundle $\sP_\cG$ in a cocycle of the form \eqref{eq:equivcocycle}.

\begin{lem} \label{lem:cocycleCechsPGDiff}
The $\Diff(\Sigma)$-equivariant structure on $\sPG\to \Bun_G(\Sigma)$ is determined by the map
$$
R_{\Diff(\Sigma)} \colon \check{C}^1(\Sigma,G)\times \Diff(\Sigma)\to \U(1),\qquad ((u,\rho),f)\mapsto [\gamma_{u,\rho}/\gamma_{f^{-1}\circ u,\rho}]\in \rmH^2(\Sigma;\U(1))\simeq \U(1).
$$
\end{lem}

\bp
The evaluation of the cocycle follows an argument analogous to the previous lemma: we compare
\begin{align*}
   \sigma((u, \rho) \cdot f) = ((f^{-1} \circ u, \rho) ,(f^{-1}\circ u, \rho, [\gamma_{f^{-1} \circ u, \rho}]), \id_{P_{f^{-1} \circ u, \rho}})
\end{align*}
and
\begin{align*}
   \sigma(u, \rho)\cdot f = ((f^{-1} \circ u, \rho), (f^{-1} \circ u, \rho, [\gamma_{u, \rho}]), \id_{P_{f^{-1} \circ u, \rho}}), 
\end{align*}
and observe that they differ up to isomorphism by the action of $[\gamma_{ u, \rho}/\gamma_{f^{-1} \circ u, \rho}]$. 
\ep

Next, we perform a similar calculation to find a cocycle for $\sPG$ relative to the cover $\cc \colon \Delta(\Sigma,G)\to \Bun_G(\Sigma)$ from Example~\ref{ex: triangulation}.
Define a section of $\sPG$ over this cover whose value at $(\tri, \rho)\in \Delta(\Sigma,G)$ is determined by the unique equivalence class $[\gamma_{\tri,\rho}]$ satisfying 
\beq\label{eq:sectiondef}
\langle \Psi_{\Sigma_\tri}, \gamma_{\tri, \rho}\rangle=1,\qquad d\gamma_{\tri,\rho}=\rho^*\alpha, \quad \Psi_{\Sigma_\tri}\in Z_2^{\rm cellular}(\Sigma;\U(1))
\eeq
where $\Psi_{\Sigma_\tri}$ is the cycle representative of the fundamental class for $\Sigma$ (in cellular cohomology) determined by the triangulation $\tri$. The existence and uniqueness of this class come from the free and transitive action of $\U(1)$ (identified with ${\rm H}^2(\Sigma; \U(1))$ via the pairing with $[\Psi_{\Sigma_\tri}]$) on the fiber of $\sPG$ over the bundle determined by the pair $(\tri, \rho)$. Hence~\eqref{eq:sectiondef} determines a functor
\begin{align}\label{eq:trivsigma2}
\sigma\colon \Delta(\Sigma,G)\to \cc^*\sPG,\qquad     \sigma(\tri, \rho) = ((\tri, \rho),(u_\tri, \rho, [\gamma_{\tri,\rho}]), \id_{u_\tri, \rho})\in \cc^*\sPG.
\end{align}
This section is compatible with pullbacks: if $T$ is a refinement of the triangulation $\tri$, we obtain a natural map $v \colon Y_T \to Y_\tri$, and the equivalence class of $v^*\gamma_{\tri, \rho}$ is equal to that of $\gamma_{T, v^*\rho}$. 
The proof of the following is a computation completely analogous to those in the proofs of Lemmas~\ref{lem:cocycleCechsPG} and~\ref{lem:cocycleCechsPGDiff}.

\begin{lem} \label{lem:triangulationcocycle}
The cocycle~\eqref{eq: U(1) cocycle} for the $\U(1)$-bundle $\sPG\to \Bun_G(\Sigma)$ relative to the section~\eqref{eq:trivsigma2} is
$$
R\colon \Delta(\Sigma,G)\times_{\Bun_G(\Sigma)}\Delta(\Sigma,G) \to \U(1),\quad R(\varphi_{Y_T,h})=\langle\Psi_{\Sigma_T},  {}^h\gamma_{T, v_1^*\rho_1}\rangle\in \U(1)
$$
where $T$ is a common refinement of triangulations $\tri_1, \tri_2$, inducing the refinements $v_i\colon Y_T \to Y_{\tri_i}$ of open covers, and $h \colon Y_T \to G$ provides an isomorphism $v_1^* \rho_1 \to v_2^* \rho_2.$ The $\Diff(\Sigma)$-equivariant structure \eqref{eq:equivcocycle} is given by
$$
R_{\Diff(\Sigma)} \colon \Delta(\Sigma,G)\times \Diff(\Sigma)\to \U(1),\qquad ((\tri,\rho),f)\mapsto \langle \Psi_{\Sigma_{f^{-1}(\tri)}}, f_\tri^*\gamma_{\tri,\rho}\rangle^{-1} \in  \U(1),
$$
using the notation for the right $\Diff(\Sigma)$-action on cochains as in Equation~\eqref{Eq:triangulationDiff}.
\end{lem}

As a final description of $\sPG$, we calculate its cocycle relative to the cover $\cc \colon \Hom(\pi_1\Sigma, G)\to \Bun_G(\Sigma)$ of Example~\ref{ex: hol pres}. We fix a cycle representative $\Psi_\Sigma^{\rm grp}$ of the fundamental class in group cohomology, which provides a well-defined pairing with the objects $(\rhohol, [\tilde\gamma])\in \sPG^\hol$: the pairing $\langle \Psi_\Sigma^{\rm grp},\tilde{\gamma}\rangle$ is independent of the choice of representative $\tilde{\gamma}$ of $[\tilde\gamma]$. Then for each $\rhohol \in \Hom(\pi_1 \Sigma, G)$, there is a unique class $(\rhohol,[\gammahol_\rhohol])$ such that $\langle \Psi_\Sigma^{\rm grp},\gammahol_\rhohol\rangle =1$. We use this to fix a section 
\begin{align}\label{eq: holonomy section}
    \sigma \colon \Hom(\pi_1 \Sigma, G) \to \cc^* \sPG^\hol, \qquad \sigma(\rhohol) = (\rhohol, (\rhohol, [\gammahol_\rhohol]), \id_{P_{\rhohol}}) \in \cc^*\sP_\cG^\rhohol.
\end{align}

\begin{rmk}\label{rmk:fundamental cycle}
One can make the choice of $\Psi_\Sigma^{\rm grp} $very concrete: for $a_i,b_i$ cycle representatives of a symplectic basis of ${\rm H}_1(\Sigma)$, we have the cycle representative of the fundamental class in the bar complex that in the case of genus $1$ is
$$
\Psi_\Sigma^{\rm grp} =a \otimes b - b \otimes a \in Z_2^{\rm grp}(\pi_1\Sigma;\Z)\subset \Z[\pi_1\Sigma]\otimes \Z[\pi_1\Sigma].
$$
The pairing $\langle \Psi_\Sigma^{\rm grp},\gammahol\rangle\in \U(1)$ defines the invariant of the principal $\cG$-bundle in Theorem~\ref{thm1} as a categorified relation in a presentation of the fundamental group of $\Sigma$, i.e., as an isomorphism in $\cG$ lifting the standard relation $\prod_{i=1}^{\sf g} [g_i,h_i]= 1$ in $\pi_1\Sigma$, see Figure~\ref{fig:4gons}.
\end{rmk}

The values of the cocycle associated to this section can then be calculated analogously to Lemmas \ref{lem:cocycleCechsPG} and \ref{lem:cocycleCechsPGDiff}.

\begin{lem}\label{lem: holonomy cocycle} The cocycles~\eqref{eq: U(1) cocycle},~\eqref{eq:equivcocycle} for the $\U(1)$-bundle $\sPG\to \Bun_G(\Sigma)$ relative to the section~\eqref{eq: holonomy section} are
$$
R\colon \Hom(\pi_1\Sigma,G) \times G \to \U(1),\quad R(\tilde\rho,h)=\langle \Psi_\Sigma^{\rm grp}, {}^h\tilde\gamma\rangle\in \U(1)
$$
$$
R_{\Diff(\Sigma)} \colon \Hom(\pi_1\Sigma, G)\times \Diff(\Sigma)\to \U(1),\qquad R_{\Diff(\Sigma)}(\rhohol,f)= \langle \Psi_{\Sigma}^{\rm grp}, f^*\tilde\gamma\rangle^{-1} \in  \U(1)
$$
identifying the fibered product as in Example \ref{ex: hol pres}, $\Hom(\pi_1\Sigma,G)\times_{\Bun_G(\Sigma)} \Hom(\pi_1\Sigma,G) \simeq \Hom(\pi_1\Sigma,G) \times G$.
\end{lem}

\section{An isomorphism between $\sPG$ and the Freed--Quinn line bundle}

\subsection{The Freed--Quinn line bundle}\label{sec:mappingcyl}

Freed and Quinn construct a line bundle over $\Bun_G(\Sigma)$ \cite[Proposition B.1]{FQ}; their construction can be phrased in terms of a cocycle for a line bundle over the triangulation presentation of $\Bun_G(\Sigma)$ from Example~\ref{ex: triangulation}. 
Below we use the notation $\Buntri$ to denote this presentation. We recall that the action of the diffeomorphism group on $\Buntri$ is inherited from the action on $\Delta(\Sigma,G)$, see~\eqref{Eq:triangulationDiff}.

We use the following conventions for the mapping cylinder of a diffeomorphism $f\colon \Sigma \to \Sigma $, 
\beq\label{eq:cylfdefn}
\Cyl(f)=(([0,1]\times \Sigma)\coprod \Sigma)/\sim\qquad (1,x)\sim f(x).
\eeq
This gives a $3$-manifold with boundary $\partial \Cyl(f) = ([0] \times \Sigma ) \coprod \Sigma$. We use the following notation for the inclusion of the two boundary components:
\begin{align*}
    \iin \colon \Sigma \to \Cyl(f), & \qquad x \mapsto (0,x);\\
    \iout \colon \Sigma \to \Cyl(f), & \qquad x \mapsto y \sim (1, f^{-1}(x)).
\end{align*}
We also have a projection
\begin{align*}
    p \colon \Cyl(f) & \to \Sigma\\
    (t,x) & \mapsto f(x),
\end{align*}
which satisfies $p \circ \iin = f$ and $p \circ \iout = \id_\Sigma$. Consider the open cover
\begin{align}\label{eq: cover of [0,1]}
    Z= [0,1) \coprod (0,1]
\end{align}
of $[0,1]$. It induces an open cover of $\Cyl(f)$ with as follows:
\begin{align}\label{eq: cover of cyl}
\begin{array}{ll}
    \lambda_0 \colon [0,1) \times \Sigma \hookrightarrow \Cyl(f), \qquad & (t,x) \mapsto (t,x); \\
    \lambda_1 \colon (0,1] \times \Sigma \hookrightarrow \Cyl(f), \qquad & (t,x) \mapsto (t, f^{-1}(x)).
\end{array}
\end{align}

With this notation established, we
turn to the Freed--Quinn construction. They construct an equivariant line bundle over $\Buntri$, i.e. a functor $\FQ\colon\Buntri\sq\Diff(\Sigma)\to \Line$; this is equivalent to a $\Diff(\Sigma)$-equivariant $U(1)$-bundle, as explained in Proposition \ref{prop: line bundles are U(1) bundles}. The functor is defined as follows: for each object of $\Buntri$ take the trivial line, i.e., $\FQ((\tri,\rho)):=\C$. Recall that the morphisms in the quotient groupoid $\Buntri\sq\Diff(\Sigma)$ consist of tuples $((T, h), f)$, where $f\colon\Sigma \to \Sigma$ is a diffeomorphism, $T$ is a mutual refinement of $\tri_1$ and $f^{-1}(\tri_2)$ and $h\colon Y_T\to G$ conjugates $\rho_1$ to $f_\tri^{*}\rho_2$ (after pulling back to $Y_T$). To match the notation in Freed--Quinn, we denote such a $((T,h),f)$ as $\tilde{f}$; it fits in the diagram below.
\beq\label{eq: diagram for f-tilde}
\begin{tikzpicture}[baseline=(basepoint)];
\node (A) at (0,0) {$P_{u_{\tri_1}, \rho_1}$};
\node (B) at (4,0) {$P_{u_{\tri_2}, \rho_2}$}; 
\node (C) at (0,-1.2) {$\Sigma$};
\node (D) at (4,-1.2) {$\Sigma.$};
\draw[->] (A) to node [above] {$\tilde{f}$} (B);
\draw[->] (A) to (C);
\draw[->] (C) to node [above] {$f$} (D);
\draw[->] (B) to  (D);
\path (0,-.75) coordinate (basepoint);
\end{tikzpicture}
\eeq
The mapping cylinder $\Cyl(\tilde{f})$ is a $G$-bundle over $\Cyl(f)$; see~\cite[pg. 8]{GanterHecke}.

The mutual refinement $T$ of triangulations yields a triangulation of~$\Cyl(f)$ compatible with the triangulations $\tri_1$ and $\tri_2$ at the boundary. The $G$-bundle $\Cyl(\tilde{f})$ trivializes on the open cover associated with this triangulation. A choice of trivialization specifies a map $\varphi_{\tilde{f}} \colon \Cyl(f)\to BG$ classifying $\Cyl(\tilde{f})$. 
Consider the pullback in cellular cohomology of the 3-cocycle $\alpha$ along $\varphi_{\tilde{f}}$, and pair this with a 3-cycle representative of the fundamental class $\Psi_{\Cyl(f)_T}\in Z_3(\Cyl(f))$ for the chosen triangulation; i.e. define 

\beq\label{eq:FQcocycle}
R_{\sf FQ}((T,h),f):=\langle \Psi_{\Cyl(f)_T}, \varphi_{\tilde{f}}^*\alpha\rangle \in \U(1),\qquad \alpha\in Z^3(BG;U(1))
\eeq
The value of the functor $\FQ$ on morphisms consists of multiplication by this element of $\U(1)$ as in \ref{subsection: cocycles}.

Because the inclusion $\iout \colon \Sigma\hookrightarrow\Cyl(f)$ is a homotopy equivalence and $\Sigma$ is 2-dimensional, the 3-cocycle $\varphi_{\tilde{f}}^*\alpha$ admits a coboundary $\Gamma'\in C^2(\Cyl(f); \U(1))$. For any such $\Gamma'$, we have that $d(\iout^*\Gamma')= \rho_2 ^* \alpha$, so that for any other cochain $\gamma_2 \in C^2(\Sigma; \U(1))$ with $d\gamma_2 = \rho_2^*\alpha$, we have a cocycle $\gamma_2/ \iout^*\Gamma'$. We can pull back this cocycle along $p_0$ and multiply by $\Gamma'$ to produce a new cochain $\Gamma$ which satisfies $d\Gamma = \varphi_{\tilde{f}}^* \alpha$ and $\iout^*\Gamma = \gamma_2$. This choice of $\Gamma$ is unique modulo exact cochains.
Thus one can also write
\begin{align*}
    R_{\sf FQ}((T,h),f) &=\langle \Psi_{\Cyl(f)_T}, \varphi_{\tilde{f}}^*\alpha\rangle\\
                        &= \langle \Psi_{\Cyl(f)_T}, d\Gamma\rangle\\
                        &= \langle \partial \Psi_{\Cyl(f)_T}, \Gamma\rangle\\
                        &= \frac{\langle \Psi_{\Sigma_{\tri_2}}, \iout^*\Gamma\rangle}{\langle \Psi_{\Sigma_{\tri_1}}, \iin^*\Gamma\rangle}\\
                        & = \frac{\langle \Psi_{\Sigma_{\tri_2}}, \gamma_2\rangle}{\langle \Psi_{\Sigma_{\tri_1}}, \iin^*\Gamma\rangle}.
\end{align*}

\subsection{Proof of Theorem~\ref{thm:main}}\label{sec:proof of main thm}
We now have two $\Diff(\Sigma)$-equivariant $\U(1)$-bundles over $\Bun_G(\Sigma)$, $\sPG$ and $\FQ$. Theorem~\ref{thm:main} claims that there is a $\Diff(\Sigma)$-equivariant isomorphism of these line bundles. Both bundles trivialize on $\Delta(\Sigma, G)$, and hence are determined by cocycles on the groupoid $\Buntri\sq\Diff(\Sigma)$.
\beq\label{eq:thecocycles}
R_{\sPG}, R_{\sf FQ}\colon \Delta(\Sigma,G)\times_{\Bun_G(\Sigma)\sq\Diff(\Sigma)}\Delta(\Sigma,G) \to \U(1)
\eeq
We will prove Theoroem~\ref{thm:main} by explicitly calculating these cocycles and finding they agree. 

As in \eqref{eq:finalcocycle}, the cocycle for each bundle breaks down into two pieces, namely one encoding the bundle over $\Bun_G(\Sigma)_\Delta$, and one encoding the $\Diff(\Sigma)$-equivariance. In this case, the first cocycle further decomposes, using the observation that any morphism $(T, h)$ in $\Buntri$ can be uniquely factored as a composition of morphisms $ (T, 1_g)$ and $(\id_\tri, h)$, i.e., a refinement of triangulations and an isomorphism of cocycles for the respective $G$-bundles. As we are checking an equality of cocycles for a $\U(1)$-bundle on a groupoid, the cocycle condition shows that Theorem~\ref{thm:main} follows from verifying the equality of cocycles on each of these more basic morphisms; this is the content of the next three lemmas. 

\begin{lem}
    For a morphism in $\Buntri$ asssociated to a refinement of triangulations, i.e. of the form $(T,1_G) \colon (\tri_1, \rho_1) \to (\tri_2, \rho_2)$, the cocycles~\eqref{eq:thecocycles} agree.
\end{lem}
\bp
The refinement $T$ of the triangulations $\tri_1, \tri_2$ induces a refinements $v_i \colon Y_T \to Y_{\tri_i}$ of the associated open covers, such that the pullbacks $v_1^* \rho_1$ and $v_2^*\rho_2$ are equal. 
By Lemma~\ref{lem:triangulationcocycle} 
\begin{align*}
    R_{\sPG}(T,1_G)&=\langle \Psi_{\Sigma_T}, \gamma_{T, v_1^*\rho_1}\rangle\\
    &= 1
\end{align*}

On the other hand, it is not hard to see that in this setting, we have $\Cyl(\id) = [0, 1] \times \Sigma$ and $\Cyl(\varphi_{\tilde f}) \cong [0,1] \times P_{u_T, v_2^*\rho_2}$ as a principal $G$-bundle over $[0,1] \times \Sigma$. This means that we can take $\Gamma = p^*\gamma_{T, v_2^* \rho_2}$, so that 
\begin{align*}
    R_{\sf FQ}(T,1_G) &= \frac{\langle \Psi_{\Sigma_{\tri_2}}, \iout^*p^* \gamma_{T, v_2^* \rho_2} \rangle}{\langle \Psi_{\Sigma_{\tri_1}}, \iin^* p^* \gamma_{T, v_2^* \rho_2}, \rangle}.
\end{align*}
Since both $p \circ \iout$ and $p \circ \iin$ are equal to $\id_\Sigma$ in this setting, we have
\begin{align*}
    R_{\sf FQ}(T, 1_G) = \frac{\langle \Psi_{\Sigma_{T}}, \gamma_{T, v_2^* \rho_2} \rangle}{\langle \Psi_{\Sigma_T}, \gamma_{T, v_2^* \rho_2}, \rangle} =1.
\end{align*}
\ep

\begin{lem}
For a morphism in $\Buntri$ of the form $(\id_\tri, h) \colon (\tri, \rho_1) \to (\tri, \rho_2)$, the cocycles~\eqref{eq:thecocycles} agree.
\end{lem}

\bp
By Lemma~\ref{lem:triangulationcocycle},
\begin{align*}
    R_{\sPG}(\id_\tri,h) = \langle \Psi_{\Sigma_\tri}, \,^h\gamma_{\tri,\rho_1}\rangle.\\
\end{align*}
We now calculate the value of $R_{\sf FQ}$. Note that for the morphism $(\id_\tri, h)$, the notation \eqref{eq: diagram for f-tilde} simplifies to $f = \id_\Sigma$ (so that $\Cyl(f) = [0,1]\times \Sigma$) and $\tilde{f} = \varphi_{Y_\tri, h}$. Recalling the construction of $P_{u_\tri, \rho_1}$ and $P_{u_\tri, \rho_2}$ as quotients of $Y_\tri \times G$ from \eqref{eq:epi2} and ~\eqref{eq:epi3}, we note that the bundle $\Cyl(\varphi_{Y_\tri, h})$ has a natural section over the cover $Z \times Y_\tri$ of $[0,1] \times \Sigma$ (for $Z = [0,1) \coprod (0,1]$ as in \eqref{eq: cover of [0,1]}). This gives a cocycle $\rho_\Cyl$ determined by $\rho_1$ over $[0,1) \times Y \times_\Sigma Y$, by $\rho_2$ over $(0,1] \times Y \times_\Sigma Y$, and by $h$ over the overlap $(0,1)$. We now look for a convenient cochain $\Gamma$ with $d\Gamma = \rho_\Cyl^* \alpha$.

To do this, we use the following alternate construction of $\rho_\Cyl$. We apply Theorem \ref{thm: lifting lemma} to $\varphi_{Y_\tri, h}$ to obtain a 1-morphism of 2-group bundles $\varphi_{Y_\tri, h, \eta} \colon (u_\tri, \rho_1, \gamma_{u_\tri, \rho_1}) \to (u_\tri, \rho_2, {}^h \gamma_{u_\tri, \rho_1})$. Recall from Remark \ref{rmk:bundle as functor} that $(\rho_1, \gamma_{u_\tri, \rho_1})$ and $(\rho_2, {}^h\gamma_{u_\tri, \rho_1})$ can be interpreted as bifunctors $\check{C}(Y_\tri) \to * \sq \cG$, while $\varphi_{Y_\tri, h, \eta}$ gives a natural transformation  between them. On the other hand, this natural transformation is equivalent to a functor 
\begin{align*}
    [1]^+ \times\check{C}(Y_\tri)  \to * \sq \cG,
\end{align*}
for $[1]^+$ the two-object category with a single nonidentity morphism in each direction between the two objects. Furthermore, there is a natural epimorphism from the \v Cech groupoid $\check{C}(Z)$ to $[1]^+$. Composing, we get a functor from the \v Cech groupoid for the product cover $Z \times Y_\tri $ of $[0,1] \times \Sigma$
\begin{align*}
    \check{C}(Z \times Y_\tri) = \check{C}(Z) \times \check{C}(Y_\tri) \to [1]^+ \times \check{C}(Y_\tri) \to * \sq \cG,
\end{align*}
whose value on morphisms we can easily calculate to agree with $\rho_\Cyl$. Applying Remark \ref{rmk:bundle as functor} again, this functor is given by a pair $(\rho_\Cyl, \Gamma)$, where $\Gamma$ satisfies $d\Gamma = \rho_\Cyl^*\alpha$ as desired. By construction, $\iin^* \Gamma = \gamma_{u_\tri, \rho_1}$ and $\iout^* \Gamma = {}^h\gamma_{u_\tri, \rho_1}$, which yields
\begin{align*}
    R_{\sf FQ}(\id_\tri, h) & =  \frac{\langle \Psi_{\Sigma_{\tri}}, \iout^*\Gamma \rangle}{\langle \Psi_{\Sigma_{\tri}}, \iin^*\Gamma\rangle} \\
    & = \frac{\langle \Psi_{\Sigma_{\tri}}, {}^h\gamma_{u_\tri, \rho_1}\rangle}{\langle \Psi_{\Sigma_{\tri}}, \gamma_{u_\tri, \rho_1} \rangle} \\
    & = \langle \Psi_{\Sigma_{\tri}}, {}^h\gamma_{u_\tri, \rho_1}\rangle.
\end{align*}
\ep

\begin{lem}
For a morphism in $\Buntri\sq\Diff(\Sigma)$ associated to a diffeomorphism $f\colon \Sigma\to \Sigma$, i.e. of the form $((\id_\tri, 1_G),f) \colon (\tri, \rho) \cdot f \to (\tri, \rho)$, the cocycles~\eqref{eq:thecocycles} agree.
\end{lem}
\bp

By Lemma~\ref{lem:triangulationcocycle}, 
\begin{align*}
    R_{\sPG}((\id_\tri, e),f)&= R_{\Diff(\Sigma)}((u_\tri,\rho), f)\\
    &= \langle \Psi_{\Sigma_{f^{-1}(\tri)}}, f_\tri^*\gamma_{\tri,\rho}\rangle^{-1}.
\end{align*}

To calculate $R_{\sf FQ}((\id_\tri, 1_G),f)$, we need to choose a convenient cochain $\Gamma$ satisfying the condition $d\Gamma = \varphi_{\tilde f}^* \alpha$ on $\Cyl(f)$. We claim that in this case we can take $\Gamma = p^* \gamma_{\tri, \rho}$. To see this, we show that the condition holds over the two pieces of the open cover \eqref{eq: cover of cyl}. 
We note that $\lambda_0^* \Cyl(\tilde f) \cong [0 \times 1) \times f^*P_{u_\tri, \rho}$, so that $\lambda_0^*\varphi_{\tilde f}$ classifies the bundle given by $\proj_\Sigma^* f^* \rho$. On the other hand, we have that $p \circ \lambda_0 = f \circ \proj_\Sigma$, so that 
\begin{align*}
   \lambda_0^*d(p^*\gamma_{\tri, \rho}) & = \proj_\Sigma^* f^* d\gamma_{\tri, \rho} \\
    & = \proj_\Sigma^* f^* (\rho^* \alpha), 
\end{align*} as desired. Similarly, $\lambda_1^* \Cyl(\tilde f) \cong (0, 1] \times P_{u_\tri, \rho}$, so that $\lambda_1^*\varphi_{\tilde f}$ classifies the bundle given by $\proj_\Sigma^*(\rho^* \alpha)$.  Since $p \circ \lambda_1 = \proj_\Sigma$, we have $\lambda_1^* d(p^* \gamma_{\tri, \rho}) = p^* d\gamma_{\tri, \rho} = \proj_\Sigma^* (\rho^*\alpha)$. We conclude that we do indeed have $d\Gamma = \varphi_{\tilde f}^* \alpha$ over all of $\Cyl(f)$, so that we can use $\Gamma$ to compute $R_{\sf FQ}$. 

Here we use that for a cochain $\beta \in C^n(\Sigma;\U(1))$ defined relative to the triangulation cover $f_\tri$, the pullback $f^* \colon C^n(\Sigma; \U(1)) \to C^n(\Sigma; \U(1))$ is realized by precomposing with $f_\tri^n$; that is, $f^* \beta = f_\tri^* \beta)$. We obtain

\begin{align*}
    R_{\sf FQ}((\id_\tri, 1_G), f) & = \frac{\langle \Psi_{\Sigma_{\tri}}, \iout^*p^*\gamma_{\tri, \rho} \rangle}{\langle \Psi_{\Sigma_{f^{-1}(\tri)}}, \iin^*p^* \gamma_{\tri, \rho}\rangle} \\ 
    & = \frac{\langle \Psi_{\Sigma_{\tri}}, \gamma_{\tri, \rho} \rangle}{\langle \Psi_{\Sigma_{f^{-1} (\tri)}}, f^* \gamma_{\tri, \rho}\rangle} \\
    & = \langle \Psi_{\Sigma_{f^{-1} (\tri)}}, f^*_\tri \gamma_{\tri, \rho}\rangle^{-1}.
\end{align*}
\ep

This completes the proof of Theorem \ref{thm:main}.

\section{Klein forms from higher geometry}

Theorem~\ref{thm:main} provides new methods for calculating the action of the mapping class group on the Freed--Quinn line bundle in purely group theoretic terms (via Definition~\ref{eq: diff action on hol P}). In this section, we illustrate such computational techniques in explicit examples, with an eye towards connecting with the classical theory of Klein forms. The main takeaway is that the algebraic computations carried out are more elementary than typical approaches to the mapping class group actions in Chern--Simons theory. 

The general approach to these calculations is described in~\S\ref{sec:MCGsetup}, before we specialize to Dehn twists in~\S\ref{sec:Dehn}, and lastly the connection to Klein forms in~\S\ref{ex: gammahol example}.

\subsection{Computing characters of mapping class groups via group cohomology of $\pi_1\Sigma$}\label{sec:MCGsetup}
Identifying an object of $\Bun_G(\Sigma)$ with  a homomorphism $\rhohol \colon \pi_1 \Sigma \to G$, the $\pi_0\Diff(\Sigma)$-action on $\sPG\to \Bun_G(\Sigma)$ can be restricted along the equivalence of groupoids,
\beq\label{eq:Difforbits}
\coprod_{[\rhohol]} \pt\sq \Stab_{\pi_0\Diff(\Sigma)}(\rhohol)\stackrel{\sim}{\hookrightarrow}  \Hom(\pi_1\Sigma, G)\sq \pi_0\Diff(\Sigma)
\eeq
for stabilizer subgroups $\Stab_{\pi_0\Diff(\Sigma)}(\rhohol)<\pi_0\Diff(\Sigma)$, where the coproduct runs over the set of orbits for the $\pi_0\Diff(\Sigma)$-action. 
The diffeomorphism group action on~$\sPG$ is determined by 
\beq\label{eq:RDiffrecal}
    R_{\Diff(\Sigma)} \colon \Hom(\pi_1\Sigma, G) \times \Diff(\Sigma) \to U(1),\quad  R_{\Diff(\Sigma)}(\rhohol,f)= \langle \Psi_{\Sigma}^{\rm grp}, f^*\tilde\gamma\rangle^{-1}
\eeq
which is computed by the pairing between the group homology fundamental class $[\Psi_\Sigma^{\rm grp}]\in \rmH_2(\pi_1\Sigma;\U(1))$ and the pullback of the cochain $\gamma\in C^2(\pi_1\Sigma;\U(1))$, as shown in Lemma \ref{lem: holonomy cocycle}. Fundamental classes have well-known explicit formulas (see Remark~\ref{rmk:fundamental cycle}), making the calculation of~\eqref{eq:RDiffrecal} very concrete in any given example.  
Using~\eqref{eq:Difforbits}, the $\Diff(\Sigma)$-equivariant structure on $\sPG$ is completely determined by homomorphisms 
\beq\label{eq:stabhomo}
    \stabhom{\rhohol} \colon \Stab_{\pi_0\Diff(\Sigma)}(\rhohol) & \to &U(1)\\ \nonumber
    f & \mapsto & R_{\Diff(\Sigma)}(\rhohol, f)
\eeq
for each $\Diff(\Sigma)$-orbit in $\Hom(\pi_1\Sigma,G)$. 
Below we will compute these homomorphisms in specific cases using the description~\eqref{eq:RDiffrecal}. 

With Klein forms as our end goal, we calculate~\eqref{eq:stabhomo} in the torus case $\Sigma = \mathbb{T}^2$; the general case follows the same structure using analogous formulas for the fundamental class of $\pi_1\Sigma$. In the genus 1 case, we have
\[
\pi_0\Diff(\Sigma)\simeq \SL_2(\Z),\quad \rmH_1(\mathbb{T}^2)\simeq\pi_1\mathbb{T}^2\simeq \Z^2,\quad \eo = \begin{bmatrix}
    1 \\ 0
\end{bmatrix},\ \et = \begin{bmatrix}
    0 \\ 1
\end{bmatrix}
\]
where $\eo$ and $\et$ 
correspond to the standard symplectic basis of ${\rm H}_1(\mathbb{T}^2)$. A homomorphism $\rhohol \colon \pi_1 \Sigma \to G$ is determined by a pair of commuting elements $\rhohol\left( \eo\right), \rhohol\left( \et\right) \in G$. A $\cG$-bundle lifting this $G$-bundle is the data of a normalized cochain $\gammahol \colon \Z^2 \times \Z^2 \to U(1)$ satisfying $d\gammahol = \rhohol^*\alpha$. 

The section $\sigma$ in~\eqref{eq: holonomy section} assigns to each $\rhohol$ a cochain $\gammahol_{\rhohol}$ satisfying the additional property that $\langle\Psi_\Sigma^{\rm grp},\gammahol_{\rhohol}\rangle =1$, which in this case is the condition 
\[
\gammahol_{\rhohol}(\eo, \et) = \gammahol_{\rhohol}(\et, \eo),\quad \Psi_\Sigma^{\rm grp}= \eo \otimes \et - \et \otimes \eo
\]
using the cycle representative of the fundamental class from Remark~\ref{rmk:fundamental cycle}. 

Given $A\in \SL_2(\Z)\simeq \pi_0\Diff(\Sigma)$, the homomorphism $A^* \rhohol \colon \pi_1 \Sigma \to G$ is determined by
\beq\label{eq: A acts on rhohol}
    A^*\rhohol\left( \eo\right) = \rhohol\left( \eo \right)^a \rhohol\left( \et\right)^c \quad \text{and} \quad A^*\rhohol\left( \et \right) = \rhohol\left( \eo \right)^b \rhohol\left( \et \right)^d,\quad A = \begin{bmatrix}
    a & b \\ c & d
    \end{bmatrix}
\eeq
and similarly, the effect on a 2-cochain $\gammahol$ is 
\beq
\gammahol_A:=A^*
\gammahol \left(\begin{bmatrix}
    x_1 \\ y_1
\end{bmatrix}, \begin{bmatrix}
    x_2 \\ y_2
\end{bmatrix}\right) = \gammahol \left( A\begin{bmatrix}
    x_1 \\ y_1
\end{bmatrix}, A\begin{bmatrix}
    x_2 \\ y_2
\end{bmatrix}\right).
\eeq
 Hence, the homomorphisms~\eqref{eq:stabhomo} determining the action of the mapping class group on $\sPG$ are given by 
\begin{align}\label{eq: important ratio}
    \stabhom{\rhohol}(A) = \langle \Psi_\Sigma^\Grp, \gammahol_{{\rhohol}, A} \rangle^{-1} = \frac {\gammahol_{{\rhohol},A }(\et, \eo)}{\gammahol_{{\rhohol},A} (\eo, \et)} = \frac{\gammahol_{\rhohol}\left( \begin{bmatrix}
        b \\ d
    \end{bmatrix},\begin{bmatrix}
        a \\ c
    \end{bmatrix}\right)}{\gammahol_{\rhohol}\left( \begin{bmatrix}
        a \\ c
    \end{bmatrix},\begin{bmatrix}
        b \\ d
    \end{bmatrix}\right)} = \frac{\gammahol_{\rhohol}(b\eo+ d\et, a\eo + c\et)}{\gammahol_{\rhohol}(a\eo + c \et, b\eo + d\et)}.
\end{align}
Given an explicit formula for $\gammahol_{\rhohol}$, it is a simple matter to evaluate the above ratio and find the value of $\stabhom{\rhohol}(A)$ in $U(1)$. This will be the case in Example \ref{ex: gammahol example} below. However, even when we do not have explicit formulas for $\gammahol_{\rhohol}$, it is still possible to calculate the above ratio by indirect (though still elementary) means, e.g., by induction on the entries of $A$. We illustrate this in Proposition \ref{ex: Nora example} below. 

\subsection{Actions by Dehn twists}\label{sec:Dehn}
We recall that the mapping class group of a surface is generated by Dehn twists. Continuing the above analysis in the case $\Sigma=\mathbb{T}^2$, we will compute the action of such Dehn twists on $G$-bundles determined by
\begin{align}\label{eq: rhohol}
    \rhohol\left( \eo \right) = g, \rhohol\left( \et \right) = 1_G \in G.
\end{align}
By~\eqref{eq: A acts on rhohol}, we have $A\in \Stab_{\SL_2(\Z)}(\rhohol)<\SL_2(\Z)$ if and only if $A\in \Gamma_1(n)$ for the congruence subgroup defined by
\begin{align}\label{eq: stabilizer condition}
  a \equiv 1 \text{ (mod $n$)}, \qquad  b \equiv 0 \text{ (mod $n$)},\qquad A=\left[\begin{array}{cc} a & b \\ c & d\end{array}\right]\in \SL_2(\Z)
\end{align}
where $n={\rm ord}(g)$ is the order of $g\in G$. In particular, the $n$th Dehn twist,
\[
T^n= \begin{bmatrix}
        1 & n \\ 0 & 1
    \end{bmatrix}
\]
satisfies the conditions \eqref{eq: stabilizer condition} and hence is in the stabilizer of $\rhohol$. The following provides a streamlined proof of a result\footnote{There is a sign difference originating from differing conventions. To remove the sign, for example, switch the roles of $\eo$ and $\et$ and consider the transpose of $T^N$.} of Ganter \cite[Lemma 2.13]{GanterHecke}. 

    \begin{prop}\label{ex: Nora example}
For $\rhohol$ defined by Equation \eqref{eq: rhohol} and $A=T^n$, we have 
\begin{align}
    \stabhom{\rhohol}(T^n) = \prod_{j=0}^{n-1} \alpha(g, g^j, g)^{-1} \in U(1).
\end{align}
    \end{prop}

\begin{proof}
In fact we show that 
    \begin{align}\label{eq: Dehn twist induction}
\frac{\gammahol_{\rhohol}\left( b\eo + \et, \eo\right)}{\gammahol_{\rhohol}\left(\eo, b\eo + \et\right)} = \prod_{j=0}^{b-1} \alpha(g, g^j, g)^{-1},\qquad b \in \Z_{\ge 0}, 
    \end{align}
where the $b=n$ case recovers the desired statement. 
When $b=0$,~\eqref{eq: Dehn twist induction} follows from the defining property of $\gammahol_{\rhohol}$. Arguing by induction, assume that Equation \eqref{eq: Dehn twist induction} holds for some fixed $b \ge 0$; then the condition $d\gammahol_{\rhohol} = \rhohol^*\alpha$ applied to the triple $\left(\eo, b\eo + \et, \eo\right)$ yields the equation
\begin{align*}
    \frac{\gammahol_{\rhohol}\left( b\eo+ \et, \eo\right)\gammahol_{\rhohol}\left( \eo, (b+1)\eo + \et)\right)}{\gammahol_{\rhohol}\left( (b+1)\eo + \et,\eo \right)\gammahol_{\rhohol}\left( \eo, b\eo + \et \right)} = \alpha\left( \rhohol\left(\eo \right), \rhohol\left(b\eo + \et \right), \rhohol\left(\eo \right)\right).
\end{align*}
Rearranging, we obtain 
\begin{align*}
    \frac{\gammahol_{\rhohol}\left( (b+1) \eo + \et , \eo\right)}{\gammahol_{\rhohol}\left( \eo, (b+1)\eo + \et\right)} = \frac{\gammahol_{\rhohol}\left( b\eo + \et, \eo\right)} {\gammahol_{\rhohol}\left( \eo, b\eo + \et \right)}\alpha(g, g^b, g)^{-1},
\end{align*}
and the desired result follows. 
\end{proof}

\subsection{Klein forms from higher geometry}\label{ex: gammahol example}\label{Freed example}

Next we specialize the techniques from~\S\ref{sec:MCGsetup} to the case that $\Sigma=\mathbb{T}$ and $G=\Z/n\Z$. The description~\eqref{eq:Difforbits} is particularly explicit in this case, 
\beq\label{eq:Kleinsetup}
\Hom(\pi_1\Sigma,\Z/n\Z)\sq \SL_2(\Z)\simeq \coprod_{m|n} \pt\sq \Gamma_1(m),
\eeq
where $\Gamma_1(m)<\SL_2(\Z)$ is the congruence subgroup defined as in~\eqref{eq: stabilizer condition}. Under this equivalence, we can choose representatives of isomorphism classes of $\Z/n\Z$-bundle to take the form~\eqref{eq: rhohol}. The decomposition~\ref{eq:Kleinsetup} guarantees that $\sPG$ for $\cG$ a categorical extension of $\Z/n\Z$ determines (and is determined by) characters of congruence subgroups via~\eqref{eq:stabhomo}. Our present goal is to compute these characters in examples, extracting transformation properties for Klein forms. By naturality and restriction to subgroups $\Z/m\Z<\Z/n\Z$, it suffices to consdier the component of~\eqref{eq:Kleinsetup} with $m=n$.

Klein forms are a flavor of modular forms with level structure. Our conventions below follow the description from \cite[Proposition 4.12]{Freed_Det}, where Klein forms arise as sections of a determinant line bundle. The mapping class group action is computed via Witten's holonomy formula involving $\eta$-invariants of $\bar\partial$-operators with an explicit formula~\cite[Equation 4.14]{Freed_Det}, which we can the compare with the characters constructed by $\sPG$ extracted via~\eqref{eq:Kleinsetup}. 

\begin{rmk}\label{eq:rmkFreedspin}
In \cite[\S4]{Freed_Det}, Freed studies the square root of the determinant line, i.e., the Pfaffian, over the moduli of $\Z/n\Z$-bundles on genus 1 curves together with a choice of spin structure. The moduli spaces in this paper do not include such spin structures, and hence below we take the tensor square of Freed's Pfaffian line, i.e., the determinant line. \end{rmk}

Let $\alpha^N \colon (\Z/n\Z)^3 \to U(1)$ be the $N$th power of the cocycle from Example \ref{ex: alpha} that generates $\rmH^3(\Z/n\Z;\U(1))$. Hence, for $n \in \Z$, and representatives $j,k,l \in \{0, 1, \ldots, n-1\}\simeq\Z/n\Z$
\begin{align}\label{eq:alphaZn}
    \alpha^N(j + n\Z, k + n\Z, l + n\Z) = \left\{ \begin{array}{ll}
       1  & \text{ if }k + l < n \\
       e^{\frac{2Nj\pi i}{n} }  & \text{ if } k + l \ge n.
    \end{array}\right.  
\end{align}
 
\begin{thm}\label{thm: cyclic rho example}
For $\rhohol \colon \Z^2 \to \Z /n\Z$ defined by $\rhohol(\eo) = 1 + n\Z, \rhohol(e_2)=n\Z$, and $A = \begin{bmatrix}
    a & b \\ c & d
\end{bmatrix} \in \Stab_{\SL_2\Z}(\rhohol)$, we have
\beq\label{eq:RrhoFreed}
    \stabhom{\rhohol}(A) = e^{2 \pi i N b / n^2 }.
\eeq
\end{thm}
\begin{proof} 
We start with a general observation for a group $G$, homomorphism $\rhohol\colon \Z^2\to G$ as in \eqref{eq: rhohol}, and 3-cocycle $\alpha\in Z^3(G;\U(1))$. Define
    \begin{align}\label{eq: explicit gamma}
        \gammahol_{\rhohol} \left(x\eo + y\et, z\eo + w\et \right) = \left\{ \begin{array}{ll}
            \prod_{k=0}^{x-1} \alpha(g, g^k, g^z)&  \text{ if $x \ge 0$},\\
             \prod_{k=1}^{-x} \frac{1}{\alpha(g, g^{-k}, g^z)}& \text{ if $x \le 0$}. 
        \end{array}\right. 
    \end{align}
    It is straightforward (if somewhat tedious) to check that $d\gammahol_{\rhohol} = \rhohol^* \alpha$, and it is also immediate that $\gammahol_{\rhohol}\left( \eo, \et \right) = \gammahol_{\rhohol}\left( \et, \eo \right) = 1$. Therefore, to calculate $\stabhom{\rhohol}(A)$, one just needs to evaluate~\eqref{eq: important ratio}. 

We apply this observation to the case at hand. First we calculate the  denominator, $\gammahol_{\rhohol}(a\eo + c\et, b\eo + d\et)$ of~\eqref{eq: important ratio}. By definition of $\gammahol_{\rhohol}$, this is a product of terms of the form $\alpha^N(1 + n\Z, \pm k + n\Z, b + n\Z)$.  By Equation \ref{eq: stabilizer condition}, $b + n\Z = n\Z$, and since $\alpha^N$ is normalized, all of these terms are trivial. We conclude that the denominator is equal to $1$. 

Now we consider the numerator, $\gammahol_{\rhohol}(b\eo + d\et, a \eo + c\et)$ of~\eqref{eq: important ratio}.  If $b \ge 0$, we have 
\begin{align*}
    \gammahol_{\rhohol}(b\eo + d\et, a\eo + c\et) = \prod_{k=0}^{b-1}
\alpha(1 + n\Z, k + n\Z, a + n\Z).
\end{align*}
By Equation \eqref{eq: stabilizer condition}, $a + n\Z = 1 + n\Z$, and we can write $b = n \delta$ for some $\delta \in \Z_{\ge 0}$. Then 
\begin{align*}
    \prod_{k=0}^{b-1}
\alpha(1 + n\Z, k + n\Z, a + n\Z) = \left( \prod_{k = 0}^{n-1}\alpha(1 + n\Z, k + n\Z, 1 + n\Z)\right)^\delta. 
\end{align*}
As $k$ goes from $0$ to $n-1$, only the value $k = n-1$ satisfies the condition $k + 1 \ge n$, yielding an $\alpha^N$ term with value $e^{2\pi i N/ n}$; the remaining $\alpha^N$ terms are all trivial. Therefore, 
\begin{align*}
     \gammahol_{\rhohol}(b\eo + d\et, a\eo + c\et) = e^{2 \pi i N \delta / n}.
\end{align*}
Similarly, when $b<0$, we can write $b = n\delta$ with $\delta \in \Z_{<0}$, and can show that again
\begin{align*}
     \gammahol_{\rhohol}(b\eo + d\et, a\eo + c\et) = e^{2 \pi i N \delta / n}.
\end{align*}
We conclude the claimed value~\eqref{eq:RrhoFreed}.
\end{proof}

\begin{rmk}
    We observe that~\eqref{eq: explicit gamma} easily recovers Proposition \ref{ex: Nora example} in the case that $A=T^n$.
\end{rmk}

\begin{proof}[Proof of Theorem~\ref{thm:Klein}]
Let $\alpha\in Z^3(\Z/n\Z;\U(1))$ denote the 3-cocycle of Example \ref{ex: alpha}. First we observe that the $\Z/n\Z$-action on $\sPG$ is trivial using~\eqref{eq:betaeq} and the fact that $\alpha$ is a normalized 3-cocycle. This $\Z/n\Z$-action also trivial on the line bundle from \cite[Proposition 4.12]{Freed_Det}, e.g., using its description as a determinant line bundle. 

It remains to compare the mapping class group actions. For this it suffices to show that the character of $\Gamma_1(n)$ agrees with the one from \cite[Proposition 4.12]{Freed_Det} (up to the sign explained in Remark~\ref{eq:rmkFreedspin}).  In Freed's notation, $\rhohol$ corresponds to $u = \frac{1}{n}$ and $v = 0$, and the values $s, t$ are given by the equation 
    \begin{align*}
        \begin{bmatrix}
            s \\ t
        \end{bmatrix} =  (I-A^T) \begin{bmatrix}
            u \\ v
        \end{bmatrix}, 
    \end{align*}
    which yields $s = (1-a)/n$ and $t = -b/n$. Under this translation,  \cite[Equation 4.14]{Freed_Det} agrees with the value from~\eqref{eq:RrhoFreed}. 
\end{proof}

\begin{rmk}
   For an arbitrary group $G$, cocycle $\alpha \colon G^3 \to A$ and homomorphism $\rhohol \colon \Z^2 \to G$, Océane Perreault (a student of the second author) has computed an cochain $\gammahol_\rhohol$ analogous to that of \eqref{eq: explicit gamma}, which allows a direct calculation of the ratio \eqref{eq: important ratio}. The formulas are somewhat complicated, and since they are not needed for our applications, we omit to include them here.
\end{rmk}

\appendix
\section{$\U(1)$-bundles, line bundles, and cocycles}
We give here our conventions and basic results on line bundles and principal $\U(1)$-bundles on (discrete) groupoids. This material can be extended to smooth bundles on Lie groupoids; we do not need such bundles for this work, but we use a framework which generalizes in a straightforward way to the smooth setting. 

\subsection{Line bundles and $\U(1)$-bundles on groupoids}

\begin{defn}\label{defn:U1bundle}
A \emph{principal $\U(1)$-bundle $\sP\to \sC$} over a groupoid $\sC$ is a groupoid $\sP$ with a (weak) $\U(1)$-action with the property that for one epimorphism $\cc\colon C_0\twoheadrightarrow \sC$ the pullback $\cc^*\sP$ is equivariantly equivalent to $C_0\times \U(1)$ over $C_0$. A morphism of $\U(1)$-bundles over $\sC$ is a $\U(1)$-equivariant functor $\sP\to \sP'$ over~$\sC$. \end{defn}

(See Definition \ref{defn:weakaction} for the definition of a weak action.)

The collection of $\U(1)$-bundles on $\sC$ forms a symmetric monoidal groupoid under tensor product, denoted~$(\Bun_{\U(1)}(\sC),\otimes)$. The \emph{trivial $\U(1)$-bundle} is the product $\underline{\U(1)}:=\sC\times \U(1)$, where $\U(1)$ is viewed as a discrete category with the evident $\U(1)$-action. A \emph{trivialization} of a bundle $\sP\to \sC$ is an isomorphism $\tau\colon \underline{\U(1)}\xrightarrow{\sim} \sP$ with the trivial bundle. By the usual arguments, an isomorphism $\sP\simeq \sC\times \U(1)$ is equivalent data to a section $\sigma\colon \sC\to \sP$. Here a section is a functor $\sigma \colon \sC \to \sP$ together with a natural isomorphism $\pi \circ \sigma \to \id_\sC$; however, for a section of $\cc^*\sP$ over a cover $\cc \colon C_0 \to \sC$, the natural isomorphism is trivial because $C_0$ is a set. 

\begin{defn}\label{def: weak fibers}
    For a $\U(1)$-bundle $\pi \colon \sP \to \sC$ and an object $c \in \sC$, the \emph{fiber of $\sP$ at $c$} is the weak fiber product $\{c\} \times _\sC \sP$, denoted $\sP_c$. It carries a $\U(1)$-action induced from the action on $\sP$. We denote the objects of the fiber by $(c, p, \varphi)$, where $p \in \sP$ and $\varphi \colon \pi(p) \to c \in \sC$. A choice of object $\sigma(c)$ (provided for example by a section) in the fiber induces an equivalence
\begin{align}\label{eq: section sigma gives iso of fiber}
    \U(1) \to \sP_c, \qquad z \mapsto \sigma(c) \cdot z,
\end{align}
which we will also denote by $\sigma(c)$, by a slight abuse of notation. 
\end{defn}

\begin{rmk}
    We caution that the functor $\sP \to \sC$ is not necessarily a fibration of categories, and in particular the natural functors from strict fibers to weak fibers are in general not equivalences. 
\end{rmk}

\begin{defn}\label{def: bundle induces functor on fibers}
The universal property of weak fiber products implies that a morphism $f \colon c \to d$ in $\sC$ induces a $\U(1)$-equivariant morphism of the fibers of the $\U(1)$-bundle, which we will denote by 
$$\sP(f) \colon \sP_c \to \sP_d.$$
On objects, $\sP(f)(c, p, \varphi)=(d, p, f \circ \varphi)$.
\end{defn}

As in the setting of bundles over manifolds, principal $\U(1)$-bundles over a groupoid $\sC$ are equivalent to Hermitian line bundles over $\sC$, as we now briefly explain. 

\begin{defn}
Let $\Line$ denote the \emph{groupoid of hermitian lines}, whose objects are 1-dimensional complex vector spaces with hermitian inner product and whose maps are linear isometries. The tensor product of vector spaces endows $\Line$ with a symmetric monoidal structure $\otimes$. 
\end{defn}

\begin{rmk}\label{rmk:Lineequiv}
Using that every 1-dimensional vector space is isomorphic to $\C$, there is an equivalence of groupoids $\Line\simeq \pt\sq\U(1)$. The monoidal strucure on $\Line$ corresponds to the 2-group structure on $\pt\sq \U(1)$. 
\end{rmk}

\begin{defn}[{\cite[\S1]{FQ}}]
For a groupoid $\sC$, a \emph{hermitian line bundle} on $\sC$ is a functor $\sL\colon \sC\to \Line$, and an isomorphism of hermitian line bundles is a natural transformation of functors. 
\end{defn}

Explicitly, a line bundle is the data of a hermitian line $\sL_y$ for each object $y\in \sC$, and a linear isometry $\sL(f)\colon \sL_y\to \sL_x$ for each morphism $f\colon y\to x$ in $\sC$. These linear maps are required to be compatible with identity morphisms and composition in $\sC$. An isomorphism $\sL\to \sL'$ in $\Line(\sC)$ is a linear isomorphism $\sL_x\to \sL_x'$ for each $x\in \sC$ satisfying the natural compatibility property.

Hermitian line bundles on $\sC$ and their isomorphisms form a groupoid $\Line(\sC)$. The tensor product of lines endows $\Line(\sC)$ with a symmetric monoidal structure also denoted~$\otimes$. The monoidal unit $\underline{\C}\in \Line(\sC)$ is the \emph{trivial line} given by the constant functor which sends objects to the standard line $\C\in \Line$ and sends morphisms to the identity automorphism $1\in \U(1)= \Aut(\C)$. 

\begin{defn}
A \emph{trivialization} of a line bundle $\sL$ is a natural isomorphism $\tau\colon \underline{\C}\xrightarrow{\sim} \sL$ with the trivial line. 
\end{defn}

A functor $F\colon \sC\to \sC'$ between groupoids induces a symmetric monoidal (pullback) functor
$$
F^*\colon \Line(\sC')\to \Line(\sC).
$$
When $F\colon \sC\to \sC'$ is an equivalence of categories, the pullback functor $F^*$ induces is an equivalence between symmetric monoidal groupoids of hermitian line bundles. 

\begin{ex} 
For a given a groupoid $\sC$, using the skeletal presentation $\sC\simeq \coprod_{[x]\in \pi_0(\sC)} \pt\sq \Aut(x)$ as outlined in Example~\ref{eg:coprodofgrps}, a line bundle on $\sC$ is equivalent data to a collection of 1-dimensional representations of $\Aut(x)$ for a representative of each isomorphism class of object $x\in \sC$. Such a line bundle is trivializable if and only if each $\Aut(x)$-representation is the trivial representatation. 
\end{ex}

\begin{ex}
 Given a cover $\cc\colon C_0\twoheadrightarrow \sC$ from a discrete category $C_0$, the pullback of any line bundle on $\sC$ admits a trivialization over $C_0$, denoted $\tau\colon 
 \underline{\C} \xrightarrow{\sim}\cc^*\sL$. The original line bundle can then be described in terms of descent data for the cover $C_0\to \sC$; see \ref{subsection: cocycles} below. One can check that the cocycle data for the skeletal presentation of the previous example is exactly the data of $\Aut(x)$-representations discussed there. 
 \end{ex}

 \begin{prop}\label{prop: line bundles are U(1) bundles}
For a groupoid $\sC$, the symmetric monoidal category of hermitian line bundles over $\sC$ is equivalent to the symmetric monoidal category of $\U(1)$-bundles over $\sC$,
\beq\label{eq:lineandbundle}
(\Line(\sC),\otimes)\simeq (\Bun_{\U(1)}(\sC),\otimes). 
\eeq
This equivalence is natural in $\sC$. 
\end{prop}
\begin{proof}[Proof sketch.] 
This follows from Remark~\ref{rmk:Lineequiv}, where one pulls back the universal $\U(1)$-bundle along a functor from $\sC$ to $\Line\simeq \pt\sq\U(1)$ to obtain a $\U(1)$-bundle over $\sC$. 
\ep

\subsection{Cocycles for line bundles}\label{subsection: cocycles}
Next, we extract a formula for a cocycle for a line bundle (or the corresponding $\U(1)$-bundle) relative to a cover $\cc\colon C_0\twoheadrightarrow \sC$ as in~\eqref{eq:presentationforcover}. 

Given a Hermitian line bundle $\sL \colon \sC \to \Line$, consider its pullback along the equivalence $\cc^*\sC \to \sC$. This gives a functor with values Hermitian lines $\sL_{\cc(x)} \in \Line$ on objects, and on morphisms $(x,y, f \colon \cc(y) \to \cc(x)) \in C_0 \times_\sC \C_0$, linear maps $\sL_{\cc(y)} \to \sL_{\cc(x)}$. Choosing an isomorphism $\tau_x \colon \C \to \sL_{c(x)}$ for each $x \in C_0$ (i.e. a trivialization $\tau$ of $\cc^* \sL$), the linear maps induce unitary automorphisms of $\C$, and hence can be identified with a number $R(x, y, f) \in \U(1)$. 

Equivalently, given a principal $\U(1)$-bundle $\sP\to\sC$, a choice of trivializing section $\sigma \colon C_0 \to \sP$ determines a function $R \colon C_0 \times_\sC C_0 \to \U(1)$, where multiplication by $R(x,y,f) \in \U(1)$ corresponds to the composition
\begin{align}\label{eq: U(1) cocycle}
    \U(1) \xrightarrow{\sigma(y)} \sP_{c(y)} \xrightarrow{\sP(f)} \sP_{c(x)} \xrightarrow{\sigma(x)^{-1}} \U(1).
\end{align}
(Here we use the notation from Definitions \ref{def: weak fibers} and \ref{def: bundle induces functor on fibers}.) That is, $R(x,y,f) \in \U(1)$ is the unique element of $\U(1)$ such that $\sP(f)(\sigma(y))$ is isomorphic to $ \sigma(x)\cdot R(x,y,f)$ in $\sP_{c(x)}$. From this characterization, it is easy to see that $R$ is a \emph{normalized 2-cocycle},
$$
R(x,y,f)\circ R(y,z,g)=R(x,z,f\circ g),\qquad R(x,x,\id_{\cc(x)})=1.
$$

Conversely, given a cover $\cc \colon C_0 \to \sC$ and a normalized $2$-cocycle $R \colon C_0 \times_\sC C_0 \to \U(1)$, we obtain a Hermitian line bundle or equivalently a principal $\U(1)$-bundle. Indeed, we define a functor $\sL \colon \cc^*\sC \to \Line$ over the presentation of the groupoid $\sC$ associated to the cover $\sC$ as follows: on objects $x \in C_0$, we take $\sL_x := \C$. The morphisms of $\cc^*\sC$ are exactly the triples $(x,y, f \colon \cc(y) \to \cc(x))$, and $R(x,y,f) \in \U(1)$ gives the desired linear isomorphism $\sL_y \to \sL_x$. 

\subsection{Cocycles for equivariant line bundles}\label{subsection: equivariant cocycles}

Let $\Gamma$ be a group acting (perhaps weakly) on the right on a groupoid $\sC$. 
\begin{defn}\label{ex:Gammaequivariant}
 A \emph{$\Gamma$-equivariant line bundle (respectively, $\U(1)$-bundle) on $\sC$} is a line bundle (respectively, $\U(1)$-bundle) on the quotient groupoid $\sC\sq \Gamma$. An \emph{equivariant trivialization} is a trivialization of a line bundle (or $\U(1)$-bundle) on $\sC\sq \Gamma$. 
\end{defn}

\begin{prop}\label{prop: equivariant bundle description}
A $\Gamma$-equivariant $\U(1)$-bundle $\sP \to \sC \sq \Gamma$ on $\sC$ is equivalent to a $\U(1)$-bundle $\pi \colon \sP' \to \sC$ equipped with an action $\sP' \times \Gamma \to \sP'$ which commutes with the $\U(1)$-action, and for which the functor $\pi$-is $\Gamma$-equivariant. 
\end{prop}
\begin{proof}[Proof sketch.] 
The pullback of $\sP$ along the quotient functor $\sC \to \sC \sq \Gamma$ yields a principal $\U(1)$-bundle $\sP' \to \sC$. The following diagram 2-commutes, yielding the dashed arrow, which gives the action of $\Gamma$ on $\sP'$:
\begin{center}
    \begin{tikzcd}
   \sP' \times \Gamma \arrow[d]\arrow[dr, dashed] \arrow[r, "\text{proj}"] & \sP' \arrow[dr] &\\
   \sC \times \Gamma\arrow[dr,"\act"] &\sP' \arrow[r]\arrow[d] & \sP \arrow[d]\\
        &\sC \arrow[r]& \sC \sq \Gamma.
    \end{tikzcd}
\end{center}

\end{proof}

As suggested by this proposition, the cocycle data for an equivariant $\U(1)$-bundle or line bundle can be divided into two pieces: an (ordinary) cocycle encoding the bundle over $\sC$, and a second piece of data encoding the $\Gamma$-action. 

More precisely, suppose we are given a $\Gamma$-equivariant principal $\U(1)$-bundle $\sP$ on a groupoid $\sC$ with $\Gamma$-action. Suppose also that we have a $\Gamma$-equivariant cover $\cc \colon C_0 \to \sC$, and a trivializing section $\sigma$ of the pullback of $\sP$ along the composition $C_0 \to \sC \to \sC\sq \Gamma$.  We obtain the pullback square in groupoids over $\pt\sq \U(1)$
\beq\nonumber
\begin{tikzpicture}[baseline=(basepoint)];
\node (A) at (0,0) {$C_0$};
\node (B) at (3,0) {$C_0\sq \Gamma$}; 
\node (C) at (0,-1.2) {$\sC$};
\node (D) at (3,-1.2) {$\sC\sq \Gamma$};
\draw[->>] (A) to (B);
\draw[->>] (A) to (C);
\draw[->>] (C) to (D);
\draw[->>] (B) to  (D);
\path (0,-.75) coordinate (basepoint);
\end{tikzpicture}
\eeq
whose arrows are all essential surjections. 

The section $\sigma$ can be viewed as giving trivializations for the intermediary bundles on $C_0\sq \Gamma$ and $\sC$ for the covers given by~$C_0$, and hence we obtain two cocycles using the arguments from \ref{subsection: cocycles}:
$$
R_\Gamma \colon C_0\times \Gamma \cong C_0 \times_{C_0 \sq \Gamma} C_0 \to \U(1), \quad R \colon C_0\times_{\sC} C_0\to \U(1).
$$
The formula for $R$ is given as in \eqref{eq: U(1) cocycle}, for the bundle $\sP' \to \sC$, and $R_\Gamma$ is given by the composition 
\beq\label{eq:equivcocycle}
R_\Gamma(x,g) \colon \U(1) \xrightarrow{\sigma(x \cdot g)} \sP_{\cc (x \cdot g)} \xrightarrow{\sP(g)} \sP_{\cc(x)} \xrightarrow{\sigma(x)^{-1}} \U(1)
\eeq
for $x \in C_0,g\in \Gamma$. Here $\sP{(g)}$ comes from the (left) action of $g^{-1}$ and the equivariance of $\cc$, so that $R_\Gamma(x,g)$ is the unique element of $\U(1)$ so that there is an isomorphism between $\sigma(x \cdot g)$ and $(\sigma(x) \cdot g) \cdot R_\Gamma(x, g)$ in the fiber $\sP'_{\cc(x \cdot g)}$. Together, these determine the cocycle for the line bundle on $\sC\sq \Gamma$ relative to the epimorphism $C_0\to \sC\sq \Gamma$ via
\beq\label{eq:finalcocycle}
R_{\Gamma} \cdot R \colon C_0\times_{\sC\sq \Gamma} C_0\simeq C_0\times_{\sC} C_0\times \Gamma\to \U(1),\quad (R_\Gamma\cdot R)(f\colon y\to x\cdot g,g)=R_\Gamma(x,g)\cdot R(f).
\eeq

\bibliographystyle{amsplain}
\bibliography{references.bib}

\end{document}